\documentclass[11pt,reqno]{amsart}
\usepackage[utf8]{inputenc}
\usepackage[T1]{fontenc}
\usepackage[english]{babel}
\usepackage{graphicx}
\usepackage{amsfonts,amsmath,amssymb,amsthm}
\usepackage[colorlinks=true, urlcolor=blue, linkcolor=blue, citecolor=blue,pdftex]{hyperref}

\usepackage[a4paper,body={160mm,235mm},centering]{geometry}
\usepackage{caption}
\usepackage[nodayofweek,level]{datetime}
\usepackage[section]{placeins}
\usepackage{algorithmicx}
\usepackage{algorithm}
\usepackage{algpseudocode}

\algnewcommand\algorithmicinput{\textbf{INPUT: }}
\algnewcommand\Input{\item[\algorithmicinput]}
\algnewcommand\algorithmicoutput{\textbf{OUTPUT: }}
\algnewcommand\Output{\item[\algorithmicoutput]}

\numberwithin{equation}{section}
\newtheorem{innercustomgeneric}{\customgenericname}
\providecommand{\customgenericname}{}
\newcommand{\newcustomtheorem}[2]{%
  \newenvironment{#1}[1]
  {%
   \renewcommand\customgenericname{#2}%
   \renewcommand\theinnercustomgeneric{##1}%
   \innercustomgeneric
  }
  {\endinnercustomgeneric}
}

\newcustomtheorem{customassumption}{Assumption}
\newtheorem{definition}{Definition}

\newtheorem{theorem}{Theorem}
\newtheorem{proposition}{Proposition}
\newtheorem{lemma}{Lemma}
\newtheorem{remark}{Remark}
\newcommand{\rset}{\mathbb{R}}

\newcommand{\ind}{\mathbf{1}}

\newcommand{\e}{\mathbb{E}}

\newcommand{\lp}{\mathrm{L}}
\newcommand{\m}{\mathcal}

\newcommand\sFor[2]{ \For{#1}#2\EndFor} 

\begin{document}
\title[Mean reflected SDEs with jumps]{Mean Reflected Stochastic Differential
  Equations with jumps\\}
\author[Ph. Briand]{Philippe Briand}
\address{Univ. Grenoble Alpes, Univ. Savoie Mont Blanc, CNRS, LAMA, 73000 Chambéry, France}
\email{philippe.briand@univ-smb.fr}

\author[A. Ghannoum]{Abir Ghannoum}
\address{Univ. Grenoble Alpes, Univ. Savoie Mont Blanc, CNRS, LAMA, 73000 Chambéry, France}
\address{Univ. Libanaise, LaMA-Liban, P.O. Box 37, Tripoli, Liban}
\email{abir.ghannoum@univ-smb.fr}

\author[C. Labart]{Céline Labart}
\address{Univ. Grenoble Alpes, Univ. Savoie Mont Blanc, CNRS, LAMA, 73000 Chambéry, France}
\email{celine.labart@univ-smb.fr}

\date{March 23, 2018}

\begin{abstract}
This paper is devoted to the study of reflected Stochastic Differential
Equations with jumps when the constraint is not on the paths of the solution
but acts on the law of the solution. This type of reflected equations have
been introduced recently by Briand, Elie and Hu \cite{BEH16} in
the context of BSDEs, when no jumps occur. In \cite{BCGL17}, the
authors study a numerical scheme based on particle systems to approximate
these reflected SDEs. In this paper, we prove existence and uniqueness of
solutions to this kind of reflected SDEs with jumps and we generalize
the results obtained in \cite{BCGL17} to this context.

\end{abstract}

\maketitle

\section{Introduction}

Reflected stochastic differential equations have been introduced in the
pionneering work of Skorokhod (see \cite{Sko61}), and their numerical approximations by Euler
schemes have been widely studied (see \cite{Slo94}, \cite{Slo01},
\cite{Lep95}, \cite{Pet95}, \cite{Pet97}). Reflected stochastic differential
equations driven by a Lévy process have also been studied in
the literature (see \cite{MR85}, \cite{KH92}). More recently, reflected
backward stochastic differential equations with jumps have been introduced and studied
(see \cite{HO03}, \cite{EHO05}, \cite{HH06}, \cite{Ess08}, \cite{CM08}, \cite{QS14}), as well as their
numerical approximation (see \cite{DL16a} and \cite{DL16b}). The main
particularity of our work comes from the fact that the constraint acts on the law of
the process $X$ rather than on its paths. The study of such equations is
linked to the mean field games theory, which has been introduced by Lasry and
Lions (see \cite{LL07a}, \cite{LL07b}, \cite{LL06a}, \cite{LL06b}) and whose probabilistic point of view is studied in \cite{CD18a}
and \cite{CD18b}. Stochastic differential equations
with mean reflection have been
introduced by Briand, Elie and Hu in their backward forms in \cite{BEH16}. In
that work, they show that mean reflected stochastic processes exist
and are uniquely defined by the associated system of equations of the following form:
\begin{equation}\label{eq:main2}
	\begin{cases}
	\begin{split}
	& X_t  =X_0+\int_0^t b(X_{s}) ds + \int_0^t \sigma(X_{s}) dB_s  + K_t,\quad t\geq 0, \\
	& \e[h(X_t)] \geq 0, \quad \int_0^t \e[h(X_s)] \, dK_s = 0, \quad t\geq 0.
	\end{split}
	\end{cases}
  \end{equation}
  Due to the fact that the reflection process $K$ depends on the law of the
  position, the authors of \cite{BCGL17}, inspired by mean field games, study the convergence of a numerical scheme based
  on particle systems to compute numerically solutions to \eqref{eq:main2}.\\

  In
this paper, we extend previous results to the case of jumps, i.e. we study existence and
uniqueness of solutions to the following mean reflected stochastic
differential equation (MR-SDE in the sequel)

\begin{equation}\label{eq:main}
	\begin{cases}
	\begin{split}
	& X_t  =X_0+\int_0^t b(X_{s^-}) ds + \int_0^t \sigma(X_{s^-}) dB_s  + \int_0^t\int_E F(X_{s^-},z) \tilde{N}(ds,dz) + K_t,\quad t\geq 0, \\
	& \e[h(X_t)] \geq 0, \quad \int_0^t \e[h(X_s)] \, dK_s = 0, \quad t\geq 0,
	\end{split}
	\end{cases}
\end{equation}
where $E=\mathbb{R}^*$, $\tilde{N}$ is a compensated Poisson measure
$\tilde{N}(ds,dz)=N(ds,dz)-\lambda(dz)ds$, and $B$ is a Brownian process
independent of $N$. We also propose a numerical scheme
based on a particle system to compute numerically solutions to \eqref{eq:main}
and study the rate of convergence of this scheme.
\medskip

Our main motivation for studying~\eqref{eq:main} comes from financial problems
submitted to risk measure constraints. Given any position $X$, its risk
measure $\rho(X)$ can be seen as the amount of own fund needed by the investor
to hold the position. For example, we can consider the following risk measure:
$\rho(X) = \inf\{m:\ \mathbb{E}[u(m+X)]\geq p\}$ where $u$ is a utility
function (concave and increasing) and $p$ is a given threshold (we refer the
reader to~\cite{ADEH99} and to~\cite{FS02} for more details on risk measures). Suppose that we are given a portfolio $X$ of assets whose dynamic, when
there is no constraint, follows the jump diffusion model
\begin{equation*}
d X_t = b(X_t) d t + \sigma(X_t) d B_t +\int_E F(X_{t-},z) \tilde{N}(dt,dz), \qquad t\geq 0.
\end{equation*}
 Given a risk measure $\rho$, one can
ask that $X_t$ remains an acceptable position at each time $t$. The constraint rewrites $\e \left[h(X_t)\right] \geq 0$ for $t\geq 0$ where $h=u-p$.

In order to satisfy this constraint, the agent has to add some cash in the portfolio through the time and the dynamic of the wealth of the portfolio becomes 
\begin{equation*}\label{eq:exempleport2}
  d X_t = b(X_t) d t + \sigma(X_t) d B_t +\int_E F(X_{t-},z) \tilde{N}(dt,dz)+d K_t, \qquad t\geq 0,
\end{equation*}
where $K_t$ is the amount of cash added up to time $t$ in the portfolio to balance the "risk" associated to $X_t$. Of course, the agent wants to cover the risk in a minimal way, adding cash only when needed: this leads to the Skorokhod condition $\e[h(X_t)] d K_t = 0$. Putting together all conditions, we end up with a dynamic of the form \eqref{eq:main} for the portfolio.

\medskip

The paper is organized as follows. In Section \ref{sec:EU},
we show that, under Lipschitz assumptions on $b$, $\sigma$ and $F$ and
bi-Lipchitz assumptions on $h$, the system admits a unique strong solution,
\emph{i.e.} there exists a unique pair of process $(X,K)$ satisfying system
\eqref{eq:main} almost surely, the process $K$ being an increasing and
deterministic process. Then, we show that, by adding some regularity on the
function $h$, the Stieltjes measure $dK$ is absolutely continuous with respect
to the Lebesgue measure and we obtain the explicit expression of its density. In Section \ref{sec:PMRSDE} we show that the system \eqref{eq:main} can be seen
as the limit of an interacting particles system with oblique reflection of
mean field type. This result allows to define in Section \ref{sec:NSMRSDE} an
algorithm based on this interacting particle system together with a classical
Euler scheme which gives a strong approximation of the solution of
\eqref{eq:main}. When $h$ is bi-Lipschitz, this leads to an approximation error in $L^2$-sense proportional to $n^{-1}
+ N^{-\frac{1}{2}}$, where $n$ is the number of points of
the discretization grid and $N$ is the number of particles. When $h$ is
smooth, we get an approximation error proportional to $n^{-1}
+ N^{-1}$. By the way, we improve the speed of convergence obtained in
\cite{BCGL17}. Finally, we illustrate these results numerically in Section \ref{sec:NI}.

\section{Existence, uniqueness and properties of the solution.}\label{sec:EU}
Throughout this paper, we consider the following set of assumptions.
\begin{customassumption}{(A.1)}
    \label{lip}
    {\color{white} \rule{\linewidth}{0.5mm} }
    (i) Lipschitz assumption: There exists a constant $C_p>0$, such that for all $x,x'\in\mathbb{R}$ and $p>0$, we have
    \begin{equation*}
    \mid b(x)-b(x')\mid^p+\mid \sigma(x)-\sigma(x')\mid^p+\int_E\mid F(x,z)-F(x',z)\mid^p\lambda(dz)\leq C_p\mid x-x'\mid^p.
    \end{equation*}
    (ii) The random variable $X_0$ is square integrable independent of $B_t$ and $N_t$.
\end{customassumption}

\begin{customassumption}{(A.2)}
    \label{bilip}
    {\color{white} \rule{\linewidth}{0.5mm} }
    (i) The function $h:\rset \longrightarrow \rset$ is an increasing function and there exist $0<m\leq M$ such that $$\forall x\in\rset,~\forall y\in\rset,~m|x-y|\leq|h(x)-h(y)|\leq M|x-y|.$$
    (ii) The initial condition $X_0$ satisfies: $\e [h(X_0)]\geq0$.
\end{customassumption}

\begin{customassumption}{(A.3)}
    \label{int}
    $\exists p>4$ such that $X_0$ belongs to $\lp^p$: $\e [|X_0|^p]<\infty$.
\end{customassumption}

\begin{customassumption}{(A.4)}
    \label{reg}
    The mapping $h$ is a twice continuously differentiable function with bounded derivatives.
\end{customassumption}

\subsection{Preliminary results}

Define the function
\begin{equation}
\label{H}
{H}    : \mathbb{R} \times \mathcal{P}(\mathbb{R}) \ni  (x,\nu) \mapsto \int h(x+z) \nu(dz),
\end{equation}
and the inverse function in space of $H$ evaluated at $0$, namely:
\begin{equation}
\bar{G}_0    : \mathcal{P}(\mathbb{R}) \ni \nu \mapsto \inf \{x  \in \mathbb{R} : H(x,\nu) \ge 0 \} , 
\end{equation}
as well as ${G}_0$, the positive part of $\bar{G}_0$:
\begin{equation}
{G}_0    : \mathcal{P}(\mathbb{R}) \ni \nu \mapsto \inf \{x  \ge 0 : H(x,\nu) \ge 0 \}. 
\end{equation}

We start by studying some properties of $H$ and $G_0$.

\begin{lemma}
Under \ref{bilip}, we have:
\begin{enumerate}
\item[(i)] For all $\nu$ in $\mathcal{P}(\mathbb{R})$, the mapping $H(\cdot,\nu): \mathbb{R}\ni x \mapsto H(x,\nu)$ is a bi-Lipschitz function, namely: 
    
    \begin{equation}
    \label{Hbilip}
      \forall x,y \in \mathbb{R}, m|x-y| \le |H(x,\nu)-H(y,\nu)|\le M|x-y|.
    \end{equation}
    
\item[(ii)]  For all $x$ in $\mathbb{R}$, the mapping $H(x,\cdot) : \mathcal{P}(\mathbb{R})\ni \nu \mapsto H(x,\nu)$ satisfies the following Lipschitz estimate:
    \begin{equation}
    \label{Hlip}
 \forall \nu,\nu' \in \mathcal{P}(\mathbb{R}),  |H(x,\nu)-H(x,\nu')|\le \left|\int h(x+\cdot) (d\nu-d\nu')\right|.
    \end{equation}
 \end{enumerate}   
\end{lemma}

\begin{proof}
The proof is straightforward from the definition of $H$ (see $\eqref{H}$).
\end{proof}

Note that thanks to Monge-Kantorovitch Theorem, assertion $\eqref{Hlip}$ implies that for all $x$ in $\mathbb{R}$, the function $H(x,\cdot)$ is Lipschitz continuous w.r.t. the Wasserstein-1 distance. Indeed, for two probability measures $\nu$ and $\nu'$, the Wasserstein-1 distance between $\nu$ and $\nu'$ is defined by:
$$W_1(\nu,\nu')=\sup_{\varphi~1-Lipschitz}\bigg|\int\varphi(d\nu-d\nu')\bigg|=\inf_{X\sim\nu~;~Y\sim\nu'}\e[|X-Y|].$$
Therefore
\begin{equation}
\label{Hwiss}
\forall\nu,\nu'\in\mathcal{P}(\mathbb{R}),~|H(x,\nu)-H(x,\nu')|\leq MW_1(\nu,\nu').
\end{equation}
Then, we have the following result about the regularity of $G_0$:
\begin{lemma}
\label{wiss}
Under \ref{bilip}, the mapping ${G}_0 :\mathcal{P}(\mathbb{R}) \ni \nu \mapsto {G}_0(\nu)$ is Lipschitz-continuous in the following sense:
\begin{equation}
|{G}_0(\nu)-{G}_0( \nu') | \le {1\over{m}}\left|\int h(\bar{G}_0(\nu)+\cdot) (d\nu-d\nu')\right|,
\end{equation}
where $\bar{G}_0(\nu)$ is the inverse of $H(\cdot,\nu) $ at point $0$. In particular
\begin{equation}
|{G}_0(\nu)-{G}_0( \nu') | \le {M\over{m}}W_1(\nu, \nu').
\end{equation}
\end{lemma}
\begin{proof}
The proof is given in (\cite{BCGL17}, Lemma 2.5).
\end{proof}

\subsection{Existence and uniqueness of the solution of \eqref{eq:main}}

We emphasize that existence and uniqueness results hold only under \ref{lip} which is the standard assumption for SDEs and \ref{bilip} which is the assumption used in \cite{BEH16}. The convergence of particles systems requires only an additional integrability assumption on the initial condition, namely \ref{int}. We sometimes add the smoothness assumption \ref{reg} on $h$ in order to improve some of the results.

We first recall the existence and uniqueness result of  in the case of SDEs.

\begin{definition}
A couple of processes $(X,K)$ is said to be a flat deterministic solution to \eqref{eq:main} if $(X,K)$ satisfy \eqref{eq:main} with $K$ being a non-decreasing deterministic function with $K_0=0$.
\end{definition}
Given this definition we have the following result.

\begin{theorem}
\label{thrm_exacte}
Under Assumptions $\ref{lip}$ and $\ref{bilip}$, the mean reflected SDE \eqref{eq:main} has a unique deterministic flat solution $(X, K)$. Moreover,
\begin{equation}
\label{K_t}
\forall t\geq 0,~K_t=\sup_{s\leq t} \inf\{x\geq0:\e[h(x+U_s)] \geq 0\}=\sup\limits_{s\le t} {G}_0(\mu_{s}),
\end{equation} 
where $(U_t)_{0\leq t\leq T}$ is the process defined by:
\begin{equation}
\label{U_t}
U_t=X_0+\int_0^t b(X_{s^-}) ds + \int_0^t \sigma(X_{s^-}) dB_s  + \int_0^t\int_E F(X_{s^-},z) \tilde{N}(ds,dz),
\end{equation}
and $(\mu_{t})_{0\le t\le T}$ is the family of marginal laws of $(U_{t})_{0\le t\le T}$.
\end{theorem}

\begin{proof}
The proof for the case of continuous backward SDEs is given in \cite{BEH16}. For the ease of the reader, we sketch the proof for the forward case with jumps.

Let $\hat{X}$ be a given process such that, for all $t>0$, $\e\big[\sup_{s\leq t}|\hat X_s|^2\big]<\infty$.
We set $$\hat U_t=X_0+\int_0^t b(\hat X_{s^-}) ds + \int_0^t \sigma(\hat X_{s^-}) dB_s  + \int_0^t\int_E F(\hat X_{s^-},z) \tilde{N}(ds,dz),$$ and define the function $K$ by setting 
\begin{equation}
\label{K}
K_t=\sup_{s\leq t} \inf\{x\geq0:\e[h(x+\hat U_s)] \geq 0\}=\sup\limits_{s\le t} {G}_0(\hat\mu_{s}).
\end{equation}
The function $K$ being given, let us define the process $X$ by the formula
$$X_t=X_0+\int_0^t b(\hat X_{s^-}) ds + \int_0^t \sigma(\hat X_{s^-}) dB_s  + \int_0^t\int_E F(\hat X_{s^-},z) \tilde{N}(ds,dz)+K_t.$$
Let us check that $(X, K)$ is the solution to \eqref{eq:main}. By definition of $K$, $\e[h(X_t)] \geq 0$ and we have $dK$ almost everywhere,  $$K_t=\sup_{s\leq t} \inf\{x\geq0:\e[h(x+\hat U_s)] \geq 0\}>0,$$
so that $\e[h(X_t)]=\e[h(\hat U_t+K_t)]=0~~~dK$-a.e. since $h$ is continuous and nondecreasing.\\
\\
Next, we consider the set $\mathcal{C}^2=\{X~ \mbox{càdlàg},~\e(\sup_{t\leq T}|X_s|^2)<\infty\}$ and the map $\Xi:\mathcal{C}^2\longrightarrow \mathcal{C}^2$ which associates to $\hat X$ the process $X$. Let us show that $\Xi$ is a contraction. Let $\hat X,~\hat X'\in \mathcal{C}^2$ be given, and define $K$ and $K'$ as above, using the same Brownian motion. We have from Assumption $\ref{lip}$, Cauchy-Schwartz and Doob inequality

\begin{equation*}
 \begin{aligned}
\e\bigg[\sup_{t\leq T}|X_t-X'_t|^2\bigg] & \leq 4 \e\bigg[\sup_{t\leq T}\Bigg\{\bigg|\int_0^t \bigg(b(\hat X_{s^-})-b(\hat X'_{s^-})\bigg) ds\bigg|^2 + \bigg|\int_0^t \bigg(\sigma(\hat X_{s^-})-\sigma(\hat X'_{s^-})\bigg) dB_s\bigg|^2 \\&~~~~ + \bigg|\int_0^t\int_E \bigg(F(\hat X_{s^-},z)-F(\hat X'_{s^-},z)\bigg) \tilde{N}(ds.dz)\bigg|^2 + |K_t-K'_t|^2\Bigg\}\bigg]\\
& \leq  4 \Bigg\{\e\bigg[\sup_{t\leq T}t\int_0^t \Big|b(\hat X_{s^-})-b(\hat X'_{s^-})\Big|^2 ds\bigg] + \e\bigg[\sup_{t\leq T}\bigg|\int_0^t \bigg(\sigma(\hat X_{s^-})-\sigma(\hat X'_{s^-})\bigg) dB_s\bigg|^2\bigg] \\&~~~~ + \e\bigg[\sup_{t\leq T}\bigg|\int_0^t\int_E \bigg(F(\hat X_{s^-},z)-F(\hat X'_{s^-},z)\bigg) \tilde{N}(ds,dz)\bigg|^2\bigg] + \sup_{t\leq T}|K_t-K'_t|^2\Bigg\} \\
& \leq C \Bigg\{T\e\bigg[\int_0^T \Big|b(\hat X_{s^-})-b(\hat X'_{s^-})\Big|^2 ds\bigg]+ \e\bigg[\int_0^T \Big|\sigma(\hat X_{s^-})-\sigma(\hat X'_{s^-})\Big|^2ds\bigg] \\&~~~~ + \int_0^T\int_E\e\bigg[\Big|F(\hat X_{s^-},z)-F(\hat X'_{s^-},z)\Big|^2\bigg] \lambda(dz)ds + \sup_{t\leq T}|K_t-K'_t|^2\Bigg\}\\
& \leq C\Bigg\{T^2C_1\e\bigg[\sup_{t\leq T}|\hat X_{t^-}-\hat X'_{t^-}|^2\bigg]+ TC_1\e\bigg[\sup_{t\leq T}|\hat X_{t^-}-\hat X'_{t^-}|^2\bigg] \\&~~~~ +  TC_1\e\bigg[\sup_{t\leq T}|\hat X_{t^-}-\hat X'_{t^-}|^2\bigg] + \sup_{t\leq T}|K_t-K'_t|^2\Bigg\}\\
& \leq C\bigg(T^2C_1+TC_2\bigg)\e\bigg[\sup_{t\leq T}|\hat X_{t}-\hat X'_{t}|^2\bigg]+C\sup_{t\leq T}|K_t-K'_t|^2.
 \end{aligned}
\end{equation*}
   
From the representation $\eqref{K}$ of the process $K$ and Lemma \ref{wiss}, we have that  
 
\begin{equation*}
 \begin{aligned}
 \sup_{t\leq T}|K_t-K'_t|^2 & \leq \frac{M}{m} \e\bigg[\sup_{t\leq T}|\hat U_t-\hat U'_t|^2\bigg]\\
 &\leq C(T^2C_1+TC_2)\e\bigg[\sup_{t\leq T}|\hat X_{t}-\hat X'_{t}|^2\bigg].
 \end{aligned}
\end{equation*}  
 
Therefore,

\begin{equation*}
 \begin{aligned}
  \e\bigg[\sup_{t\leq T}|X_t-X'_t|^2\bigg] & \leq  C(1+T)T\e\bigg[\sup_{t\leq T}|\hat X_{t}-\hat X'_{t}|^2\bigg].
 \end{aligned}
\end{equation*}   
 Hence, there exists a positive $\mathcal{T}$, depending on $b$, $\sigma$, $F$ and $h$ only, such that for all $T <\mathcal{T}$, the map $\Xi$ is a contraction. We first deduce the existence and uniqueness of solution on $[0, \mathcal{T}]$ and then on $\rset^+$ by iterating the construction.  
\end{proof}

\subsection{Regularity results on $K$, $X$ and $U$}

\begin{remark}
\label{rem}
Note that from this construction, we deduce that for all $0\leq s< t$: 
    \begin{align*}
     &K_{t}-K_{s}\\
   &= \sup\limits_{s\le r\le t} \inf  \left\{x\ge 0 : \mathbb{E}\left[ h \left( x+X_s+\int_s^{r} b(X_{u^-}) du +\int_s^{r} \sigma(X_{u^-}) dB_u + \int_s^{r}\int_E F(X_{u^-},z) \tilde{N}(du,dz) \right)\right] \right\}.    
    \end{align*}       
\end{remark}

\begin{proof}
From the representation $\eqref{K_t}$ of the process $K$, we have
\begin{equation*}
\begin{aligned}
K_t&=\sup_{r\leq t} \inf\bigg\{x\geq0:\e[h(x+U_r)] \geq 0\bigg\}\\&
=\sup_{r\leq t} G_0(U_r)\\&
=\max\bigg\{\sup_{r\leq s}G_0(U_r),\sup_{s\leq r\leq t} G_0(U_r)\bigg\}\\&
=\max\bigg\{K_s,\sup_{s\leq r\leq t} G_0(U_r)\bigg\}\\&
=\max\bigg\{K_s,\sup_{s\leq r\leq t} G_0(X_s-K_s+U_r-U_s)\bigg\}\\&
=\max\bigg\{K_s,\sup_{s\leq r\leq t} \bigg[\bar{G_0}(X_s-K_s+U_r-U_s)^+\bigg]\bigg\}.
\end{aligned}
\end{equation*}
By the definition of $\bar{G_0}$, we can observe that for all $y\in\mathbb{R}$, $\bar{G_0}(X+y)=\bar{G_0}(X)-y$, so we get
\begin{equation*}
\begin{aligned}
K_t&=\max\bigg\{K_s,\sup_{s\leq r\leq t} \bigg[\bigg(K_s+\bar{G_0}(X_s+U_r-U_s)\bigg)^+\bigg]\bigg\}\\&
=K_s+\max\bigg\{0,\sup_{s\leq r\leq t} \bigg[\bigg(K_s+\bar{G_0}(X_s+U_r-U_s)\bigg)^+-K_s\bigg]\bigg\}.
\end{aligned}
\end{equation*}
Note that $\sup_r (f(r)^+)=(\sup_r f(r))^+=\max(0,\sup_r f(r))$ for all function $f$, and obviously
\begin{equation*}
\begin{aligned}
K_t&=K_s+\sup_{s\leq r\leq t} \bigg[\bigg\{\bigg(K_s+\bar{G_0}(X_s+U_r-U_s)\bigg)^+-K_s\bigg\}^+\bigg]\\&
=K_s+\sup_{s\leq r\leq t}\bigg[\bigg(\bar{G_0}(X_s+U_r-U_s)\bigg)^+\bigg]\\&
=K_s+\sup_{s\leq r\leq t}G_0(X_s+U_r-U_s),
\end{aligned}
\end{equation*}
and so
\begin{equation*}
\begin{aligned}
K_t-K_s=\sup_{s\leq r\leq t}G_0(X_s+U_r-U_s).
\end{aligned}
\end{equation*}
\end{proof}
In the following, we make an intensive use of this representation formula of the process $K$.\\

Let $(\m{F}_s)_{s\geq0}$ be a filtration on $(\Omega,\m{F},\mathbb{P})$ such that $(X_s)_{s\geq0}$ is an $(\m{F}_s)_{s\geq0}$-adapted process.
\begin{proposition}
\label{propriete_1}
Suppose that Assumptions $\ref{lip}$ and $\ref{bilip}$ hold. Then, for all $p\geq 2$, there exists a positive constant $K_p$, depending on $T$, $b$, $\sigma$, $F$ and $h$ such that
\begin{enumerate}
\item[(i)] $\e\big[\sup_{t\leq T}|X_t|^p\big]\leq K_p\big(1+\e\big[|X_0|^p\big]\big).$
\item[(ii)]$\forall~0\leq s\leq t\leq T,~~~~\e\big[\sup_{s\leq u\leq t}|X_u|^p|\m{F}_s\big]\leq C\big(1+|X_s|^p\big).$

\end{enumerate}
\end{proposition}

\begin{remark}
\label{rem_U_1}
Under the same conditions, we conclude that $$\e\big[\sup_{t\leq T}|U_t|^p\big]\leq K_p\big(1+\e\big[|X_0|^p\big]\big).$$
\end{remark}

\begin{proof}[Proof of (i)]\renewcommand{\qedsymbol}{}
We have
\begin{equation*}
\begin{aligned}
\e\bigg[\sup_{t\leq T}|X_t|^p\bigg]&\leq 5^{p-1}\Bigg\{\e|X_0|^p+\e \sup_{t\leq T}\bigg(\int_0^t |b(X_{s^-})| ds\bigg)^p + \e \sup_{t\leq T}\bigg|\int_0^t \sigma(X_{s^-}) dB_s\bigg|^p  \\&~~+ \e \sup_{t\leq T}\bigg|\int_0^t\int_E F(X_{s^-},z) \tilde{N}(ds,dz)\bigg|^p + K_T^p\Bigg\}. 
\end{aligned}
\end{equation*}
Let us first consider the last term $K_T =\sup_{t\leq T}G_0(\mu_s)$. From the Lipschitz property of Lemma $\ref{wiss}$ of $G_0$ and the definition of the Wasserstein metric we have
$$\forall t\geq 0,~|G_0(\mu_t)|\leq \frac{M}{m}\e[|U_t-U_0|],$$
since $G_0(\mu_0)=0$ as $\e[h(X_0)]\geq 0$ and where $U$ is defined by $\eqref{U}$. Therefore
\begin{equation*}
\begin{aligned}
|K_T|^p=|\sup_{t\leq T}G_0(\mu_t)|^p&\leq 3^{p-1}\bigg(\frac{M}{m}\bigg)^p\Bigg\{\e \sup_{t\leq T}\bigg(\int_0^t |b(X_{s^-})| ds\bigg)^p + \e \sup_{t\leq T}\bigg|\int_0^t \sigma(X_{s^-}) dB_s\bigg|^p  \\&~~+ \e \sup_{t\leq T}\bigg|\int_0^t\int_E F(X_{s^-},z) \tilde{N}(ds,dz)\bigg|^p\Bigg\},
\end{aligned}
\end{equation*}
and so
\begin{equation*}
\begin{aligned}
\e\big[\sup_{t\leq T}|X_t|^p\big]&\leq C(p,M,m)\e\bigg[|X_0|^p+\sup_{t\leq T}\bigg(\int_0^t |b(X_{s^-})| ds\bigg)^p + \sup_{t\leq T}\bigg|\int_0^t \sigma(X_{s^-}) dB_s\bigg|^p  \\&~~+ \sup_{t\leq T}\bigg|\int_0^t\int_E F(X_{s^-},z) \tilde{N}(ds,dz)\bigg|^p\bigg].
\end{aligned}
\end{equation*}
Hence, using Assumption $\ref{lip}$, Cauchy-Schwartz, Doob and BDG inequality we get
\begin{equation*}
\begin{aligned}
\e\bigg[\sup_{t\leq T}|X_t|^p\bigg]&\le C\Bigg\{\e\bigg[|X_0|^p\bigg] + T^{p-1}\e\bigg[\int_0^T (1+|X_{s^-}|)^p ds\bigg] + C_1\e\bigg[\int_0^T (1+|X_{s^-}|)^2 ds\bigg]^\frac{p}{2} \\&~~+ C_2\e\bigg[\int_0^T (1+|X_{s^-}|)^p ds\bigg]\Bigg\}\\
&\le C_1\bigg(1+\e|X_0|^p\bigg)+C_2\int_{0}^{T}\e\bigg[\sup_{t\leq r}|X_t|^p\bigg]dr,
\end{aligned}
\end{equation*}
and from Gronwall's Lemma, we can conclude that  for all $p\geq 2$, there exists a positive constant $K_p$, depending on $T$, $b$, $\sigma$, $F$ and $h$ such that
$$\e\big[\sup_{t\leq T}|X_t|^p\big]\leq K_p\big(1+\e\big[|X_0|^p\big]\big).$$
\end{proof}
\begin{proof}[Proof of (ii)]
For the first part, we have
\begin{equation*}
\begin{aligned}
 X_u &=U_u+K_u\\&=X_s+(U_u-U_s)+(K_u-K_s)\\&=X_s+\int_s^u b(X_{r^-}) dr + \int_s^u \sigma(X_{r^-}) dB_r  + \int_s^u\int_E F(X_{r^-},z) \tilde{N}(dr,dz)\\&~~~~ + (K_u-K_s).
\end{aligned}
\end{equation*}
Define that $\e_s[\cdotp]=\e[\cdotp|\m{F}_s]$, then we get
\begin{equation*}
\begin{aligned}
\e_s\bigg[\sup_{s\leq u\leq t}|X_u|^p\bigg]&\leq 5^{p-1}\Bigg\{\e_s\bigg[|X_s|^p\bigg]+\e_s\bigg[ \sup_{s\leq u\leq t}\bigg|\int_s^u b(X_{r^-}) dr\bigg|^p\bigg] + \e_s \bigg[\sup_{s\leq u\leq t}\bigg|\int_s^u \sigma(X_{r^-}) dB_r\bigg|^p\bigg]  \\&~~+ \e_s \bigg[\sup_{s\leq u\leq t}\bigg|\int_s^u\int_E F(X_{r^-},z) \tilde{N}(dr,dz)\bigg|^p\bigg] +\bigg|K_t-K_s\bigg|^p\Bigg\}\\&
\leq  C\Bigg\{|X_s|^p+T^{p-1}\int_s^t \e_s\bigg[\bigg|b(X_{r^-})\bigg|^p\bigg] dr+ \int_s^t \e_s\bigg[\bigg|\sigma(X_{r^-})\bigg|^p\bigg] dr  \\&~~+ \int_s^t\int_E \e_s\bigg[\bigg|F(X_{r^-},z)\bigg|^p\bigg]\lambda(dz) dr +2\bigg|K_T\bigg|^p\Bigg\}\\& 
\leq C(T)\Bigg\{|X_s|^p+C_1\int_s^t \e_s\bigg[1+|X_{r^-}|^p\bigg]dr+2\bigg|K_T\bigg|^p\Bigg\}\\&
\leq C_1(1+|X_s|^p)+C_2\int_s^t \e_s[\sup_{s\leq u \leq r}
|X_{u^-}|^p]dr.
\end{aligned}
\end{equation*} 
Finally, from Gronwall's Lemma, we deduce that for all $0\leq s\leq t\leq T$, there exists a constant $C$, depending on $p$, $T$, $b$, $\sigma$, $F$ and $h$ such that$$\e\big[\sup_{s\leq u\leq t}|X_u|^p|\m{F}_s\big]\leq C\big(1+|X_s|^p\big).$$
\end{proof}

\begin{proposition}
\label{propriete_2}
Let $p\geq2$ and let Assumptions \ref{lip}, \ref{bilip} and \ref{int} hold. There exists a constant $C$ depending on $p$, $T$, $b$, $\sigma$, $F$ and $h$ such that
\begin{enumerate}
\item [(i)] $\forall~0\leq s<t\leq T,~~~~|K_t-K_s|\leq C|t-s|^{(1/2)}.$
\item [(ii)] $\forall~0\leq s\leq t\leq T,~~~~\e\big[|U_t-U_s|^p\big]\leq C|t-s|.$
\item [(iii)] $\forall~0\leq r<s<t\leq T,~~~~\e[|U_s-U_r|^p|U_t-U_s|^p]\leq C|t-r|^2.$
\end{enumerate}
\end{proposition}

\begin{remark}
\label{rem_X}
Under the same conditions, we conclude that $$\forall~0\leq s\leq t\leq T,~~~~\e\big[|X_t-X_s|^p\big]\leq C|t-s|.$$
\end{remark}

\begin{proof}[Proof of (i)]\renewcommand{\qedsymbol}{}
Let us recall that $$\bar{G_0}(X)=\inf\{x\in\mathbb{R}:\e[h(x+ X)] \geq 0\},$$
$$G_0(X)=(\bar{G_0}(X))^+=\inf\{x\geq 0:\e[h(x+X)] \geq 0\},$$
for all process $X$.\\
From Remark \ref{rem}, we have
\begin{equation}
\label{Kt-Ks}
\begin{aligned}
K_t-K_s=\sup_{s\leq r\leq t}G_0(X_s+U_r-U_s).
\end{aligned}
\end{equation}
Hence, from the previous representation of $K_t-K_s$, we will deduce the $\frac{1}{2}$-Hölder property of the function $t\longmapsto K_t$. Indeed, since by definition $G_0(X_s)=0$, if $s<t$, using Lemma \ref{wiss},
\begin{equation*}
\begin{aligned}
|K_t-K_s|&=\sup_{s\leq r\leq t}G_0(X_s+U_r-U_s)\\&
=\sup_{s\leq r\leq t}[G_0(X_s+U_r-U_s)-G_0(X_s)]\\&
=\frac{M}{m}\sup_{s\leq r\leq t}\e[|U_r-U_s|],
\end{aligned}
\end{equation*}
and so
\begin{equation*}
\begin{aligned}
|K_t-K_s|&\leq C\Bigg\{\e \bigg[\sup_{s\leq r\leq t}\bigg|\int_s^r b(X_{u^-}) du\bigg|\bigg] + \bigg(\e \bigg[\sup_{s\leq r\leq t}\bigg|\int_s^r \sigma(X_{u^-}) dB_u\bigg|^2\bigg]\bigg)^{1/2}\\&~~~~ + \bigg(\e \bigg[\sup_{s\leq r\leq t}\bigg|\int_s^r\int_E F(X_{u^-},z) \tilde{N}(du,dz)\bigg|^2\bigg]\bigg)^{1/2}\Bigg\}\\&
\leq C\Bigg\{\int_s^t\e\bigg[\bigg|b(X_{u^-})\bigg|\bigg]du+\bigg(\e\bigg[\int_s^t\bigg|\sigma(X_{u^-})\bigg|^2du\bigg]\bigg)^{1/2}\\&~~~~+\bigg(\e \bigg[\int_s^t\int_E \bigg|F(X_{u^-},z) \bigg|^2\lambda(dz)du\bigg]\bigg)^{1/2}\Bigg\}\\&
\leq C\Bigg\{|t-s|\e\bigg[1+\sup_{u\leq T}|X_u|\bigg]+|t-s|^{1/2}\bigg(\e\bigg[1+\sup_{u\leq T}|X_u|^2\bigg]\bigg)^{1/2}\Bigg\}.
\end{aligned}
\end{equation*}
Therefore, if $X_0\in \lp^p$ for some $p\geq2$, it follows from Proposition \ref{propriete_1} that 
\begin{equation*}
\begin{aligned}
|K_t-K_s|\leq C|t-s|^{1/2}.
\end{aligned}
\end{equation*}

\end{proof}
\begin{proof}[Proof of (ii)]\renewcommand{\qedsymbol}{}
    \begin{equation*}
    \begin{aligned}
    \e\bigg[|U_t-U_s|^p\bigg]&\leq 4^{p-1}\e\Bigg[\bigg(\int_s^t |b(X_{r^-})| dr\bigg)^p + \bigg|\int_s^t \sigma(X_{r^-}) dB_r\bigg|^p  \\&~~+\bigg|\int_s^t\int_E F(X_{r^-},z) \tilde{N}(dr,dz)\bigg|^p \Bigg]\\
    &\leq C\sup\limits_{0\le r\le t}\e\bigg[\bigg(\int_s^r|b(X_{u^-})| du\bigg)^p + \bigg|\int_s^r \sigma(X_{u^-}) dB_u\bigg|^p  \\&~~+\bigg|\int_s^r\int_E F(X_{u^-},z) \tilde{N}(du,dz)\bigg|^p\bigg]\\
    &\leq C\Bigg\{|t-s|^{p-1}\e\bigg[\int_s^t (1+|X_{u^-}|)^p du\bigg] + C_1\e\bigg[\bigg(\int_s^t (1+|X_{u^-}|)^2 du\bigg)^{p/2}\bigg] \\&~~+ C_2\e\bigg[\int_s^t (1+|X_{u^-}|)^p du\bigg]\Bigg\}\\
    & \leq C_1\e\bigg[1+\sup_{t\leq T}|X_t|^p\bigg]|t-s|^p+C_2\e\bigg[\bigg(1+\sup_{t\leq T}|X_t|^2\bigg)^{p/2}\bigg]|t-s|^{p/2}\\&~~+C_3\e\bigg[1+\sup_{t\leq T}|X_t|^p\bigg]|t-s|.
    \end{aligned}
    \end{equation*}
    Finally, if $X_0\in \lp^p$ for some $p\geq2$, we can conclude that there exists a constant $C$, depending on $p$, $T$, $b$, $\sigma$, $F$ and $h$ such that
    $$\forall~0\leq s\leq t\leq T,~~~~\e\big[|X_t-X_s|^p\big]\leq C|t-s|.$$
\end{proof}
\begin{proof}[Proof of (iii)]
Let $0\leq r<s<t\leq T$, we have
\begin{equation*}
\begin{aligned}
\e\bigg[|U_s-U_r|^p|U_t-U_s|^p\bigg]&
\leq \e\bigg[|U_s-U_r|^p\e_s[|U_t-U_s|^p]\bigg]\\&
\leq C\e\Bigg[|U_s-U_r|^p\Bigg\{\e_s\bigg[\bigg|\int_s^t b(X_{s^-}) ds\bigg|^p\bigg]+\e_s\bigg[\bigg|\int_s^t \sigma(X_{s^-}) dB_s\bigg|^p\bigg]\\&~~~~+\e_s\bigg[\bigg|\int_s^t\int_E F(X_{s^-},z) d\tilde{N}(ds,dz)\bigg|^p\bigg]\Bigg\}\Bigg].
\end{aligned}
\end{equation*}
Then, from Burkholder-Davis-Gundy inequality, we get
\begin{equation*}
\begin{aligned}
\e\bigg[|U_s-U_r|^p|U_t-U_s|^p\bigg]&
\leq C\e\Bigg[|U_s-U_r|^p\Bigg\{\e_s\bigg[\bigg|\int_s^t b(X_{s^-}) ds\bigg|^p\bigg]+\bigg(\e_s\bigg[\int_s^t \bigg|\sigma(X_{s^-})\bigg|^2 ds\bigg]\bigg)^{p/2}\\&~~~~+\e_s\bigg[\int_s^t\int_E \bigg|F(X_{s^-},z)\bigg|^p \lambda(dz)ds\bigg]\Bigg\}\Bigg]\\&
\leq C\e\Bigg[|U_s-U_r|^p\Bigg\{|t-s|^p\bigg(1+\e_s\bigg[\sup_{s\leq u\leq t}|X_u|^p\bigg]\bigg)\\&~~~~+|t-s|^{p/2}\bigg(1+\e_s\bigg[\sup_{s\leq u\leq t}|X_u|^2\bigg]^{p/2}\bigg)+|t-s|\bigg(1+\e_s\bigg[\sup_{s\leq u\leq t}|X_u|^p\bigg]\bigg)\Bigg\}\Bigg]\\&
\leq C\e\Bigg[|U_s-U_r|^p\Bigg\{|t-s|\bigg(1+\e_s\bigg[\sup_{s\leq u\leq t}|X_u|^p\bigg]\bigg)\Bigg\}\Bigg],
\end{aligned}
\end{equation*}
thus, from (i) and Proposition \ref{propriete_1}, we obtain
\begin{equation*}
\begin{aligned}
\e\bigg[|U_s-U_r|^p|U_t-U_s|^p\bigg]&
\leq C_1|t-s|\e\Bigg[|U_s-U_r|^p\Bigg]+C_2|t-s|\e\Bigg[|U_s-U_r|^p|X_s|^p\Bigg]\\&
\leq C_1|t-s||s-r|+C_2|t-s|\e\Bigg[|U_s-U_r|^p\bigg(|X_s-X_r|^p+|X_r|^p\bigg)\Bigg]\\&
\leq C_1|t-r|^2+C_2|t-s|\e\Bigg[2^{p-1}|U_s-U_r|^{p}\bigg(|U_s-U_r|^p+|K_s-K_r|^p\bigg)\Bigg]\\&~~~~+C_3|t-s|\e\Bigg[|U_s-U_r|^p|X_r|^p\Bigg]\\&
\leq C_1|t-r|^2+C_2|t-s|\e\Bigg[|U_s-U_r|^{2p}\Bigg]+C_3|t-s||s-r|^{p/2}\e\Bigg[|U_s-U_r|^{p}\Bigg]\\&~~~~+C_4|t-s|\e\Bigg[|U_s-U_r|^p|X_r|^p\Bigg]\\&
\leq C_1|t-r|^2+C_4|t-s|\e\Bigg[|X_r|^p\e_r[|U_s-U_r|^p]\Bigg].
\end{aligned}
\end{equation*}
Following the proof of (ii), we can also get
\begin{equation*}
\begin{aligned}
\e_r[|U_s-U_r|^p]\leq C|s-r|\bigg(1+\e_r\bigg[\sup_{r\leq u\leq s}|X_u|^p\bigg]\bigg).
\end{aligned}
\end{equation*}
Then
\begin{equation*}
\begin{aligned}
\e\bigg[|U_s-U_r|^p|U_t-U_s|^p\bigg]&
\leq C_1|t-r|^2+C_2|t-s||s-r|\e\Bigg[|X_r|^p\bigg(1+\e_r\bigg[\sup_{r\leq u\leq s}|X_u|^p\bigg]\bigg)\Bigg]\\&
\leq C_1|t-r|^2+C_2|t-r|^2\e\Bigg[|X_r|^p\bigg(1+\sup_{r\leq u\leq s}|X_u|^p\bigg)\Bigg].
\end{aligned}
\end{equation*}
Under \ref{int}, we conclude that
\begin{equation*}
\e[|U_s-U_r|^p|U_t-U_s|^p]\leq C|t-r|^2,~~~~\forall~0\leq r<s<t\leq1.
\end{equation*}
\end{proof}

\subsection{Density of $K$}
Let $\m{L}$ be the linear partial operator of second order defined by
\begin{equation}
\label{L}
\m{L}f(x):=b(x)\frac{\partial}{\partial x}f(x)+\frac{1}{2}\sigma\sigma^*(x)\frac{\partial^2}{\partial x^2}f(x)+\int_E\bigg(f\big(x+F(x,z)\big)-f(x)-F(x,z)f'(x)\bigg)\lambda(dz),
\end{equation}
for any twice continuously differentiable function $f$.
\begin{proposition}
\label{propdensite}
Suppose that Assumptions $\ref{lip}$, $\ref{bilip}$ and $\ref{reg}$ hold and let $(X,K)$ denote the unique deterministic flat solution to \eqref{eq:main}. Then the Stieljes measure $dK$ is absolutely continuous with respect to the Lebesgue measure (Proposition \ref{propriete_2}) with density
\begin{equation}
\label{densite}
k:\mathbb{R}^+\ni t\longmapsto\frac{(\e[\m{L}h(X_{t^-})])^-}{\e[h'(X_{t^-})]}{\bf{1}}_{\e[h(X_t)]=0}.
\end{equation}
\end{proposition}

Let us admit for the moment the following result that will be useful for our proof.
\begin{lemma}
\label{lemmacont}
The functions $t\longmapsto\e\left[h(X_t)\right]$ and $t\longmapsto\e \left[\m{L}h(X_t)\right]$ are continuous.
\end{lemma}

\begin{lemma}
\label{lemmacont_generale}
If $\varphi$ is a continuous function such that, for some $C\geq 0$ and $p\geq 1$,
$$
\forall x\in\mathbb{R}, \quad |\varphi(x)|\leq C(1+|x|^p),
$$ 
then the function $t\longmapsto\e [\varphi(X_t)]$ is continuous.
\end{lemma}
The proof of Lemma \ref{lemmacont_generale} is given in Appendix A.1. We may now proceed to the proof of Proposition \ref{propdensite}.
\begin{proof}
Let $t$ in $[0,T]$. For all positive r, we have
\begin{equation*}
\begin{aligned}
X_{t+r}  &=X_t+\int_t^{t+r} \bigg(b(X_{s^-})-\int_E F(X_{s^-},z)\lambda(dz)\bigg) ds + \int_t^{t+r} \sigma(X_{s^-}) dB_s  + \int_t^{t+r}\int_E F(X_{s^-},z) N(ds,dz) \\&~~~~+ K_{t+r}-K_t.
\end{aligned}
\end{equation*}
Under \ref{reg} and thanks to It\^o's formula we get
\begin{equation*}
\begin{aligned}
h(X_{t+r})-h(X_t)&=\int_t^{t+r}b(X_{s^-})h'(X_{s^-}) ds+\int_t^{t+r} \sigma(X_{s^-})h'(X_{s^-}) dB_s+\int_t^{t+r}\int_E F(X_{s^-},z)h'(X_{s^-}) \tilde{N}(ds,dz)\\&~~~~+\int_t^{t+r} h'(X_{s^-}) dK_s+\frac{1}{2}\int_t^{t+r} \sigma^2(X_{s^-})h''(X_{s^-}) ds\\&~~~~+\int_t^{t+r}\int_E\bigg(h\big(X_{s^-}+F(X_{s^-},z)\big)-h(X_{s^-})-F(X_{s^-},z)h'(X_{s^-})\bigg)N(ds,dz)\\&
= \int_t^{t+r}b(X_{s^-})h'(X_{s^-}) ds+\int_t^{t+r} \sigma(X_{s^-})h'(X_{s^-}) dB_s+\int_t^{t+r}\int_E F(X_{s^-},z)h'(X_{s^-}) \tilde{N}(ds.dz)\\&~~~~+\int_t^{t+r} h'(X_{s^-}) dK_s+\frac{1}{2}\int_t^{t+r} \sigma^2(X_{s^-})h''(X_{s^-}) ds\\&~~~~+\int_t^{t+r}\int_E\bigg(h\big(X_{s^-}+F(X_{s^-})\big)-h(X_{s^-})-F(X_{s^-})h'(X_{s^-})\bigg)\lambda(dz)ds\\&~~~~+\int_t^{t+r}\int_E\bigg(h\big(X_{s^-}+F(X_{s^-},z)\big)-h(X_{s^-})-F(X_{s^-},z)h'(X_{s^-})\bigg)\tilde{N}(ds,dz)\\&
= \int_t^{t+r}\m{L}h(X_{s^-}) ds+\int_t^{t+r} h'(X_{s^-}) dK_s+\int_t^{t+r} \sigma(X_{s^-})h'(X_{s^-}) dB_s\\&~~~~+\int_t^{t+r}\int_E\bigg(h\big(X_{s^-}+F(X_{s^-},z)\big)-h(X_{s^-})\bigg)\tilde{N}(ds,dz),
\end{aligned}
\end{equation*}
where $\m{L}$ is given by $\eqref{L}$. Thus, we obtain
\begin{equation}
\label{ito}
\e\bigg(\int_t^{t+r} h'(X_{s^-}) dK_s\bigg)=\e h(X_{t+r})-\e h(X_t)-\int_t^{t+r}\e\m{L}h(X_{s^-}) ds.
\end{equation}

Suppose, at the one hand, that $\e h(X_t)>0$. Then, by the continuity of $t\longmapsto\e h(X_t)$, we get that there exists a positive $\m{R}$ such that for all $r\in[0,\m{R}]$, $\e h(X_{t+r})>0$. This implies in particular, from the definition of $K$, that $dK([t,t+r])=0$ for all $r$ in $[0,\m{R}]$.

At the second hand, suppose that $\e h(X_t)=0$, then two cases arise. Let us first assume that $\e \m{L}h(X_t)>0$. Hence, we can find a positive $\m{R}'$ such that for all $r\in[0,\m{R}']$, $\e \m{L}h(X_{t+r})>0$. We thus deduce from our Assumptions and $\eqref{ito}$ that $\e h(X_{t+r})>0$ for all $r\in(0,\m{R}']$. Therefore,
$dK([t,t+r])=0$ for all $r\in[0,\m{R}']$. Suppose next that $\e \m{L}h(X_t)\leq0$. By continuity of $t\longmapsto\e \m{L}h(X_t)$, there exists a positive $\m{R}''$ such that for all $r\in[0,\m{R}'']$ it holds $\e \m{L}h(X_{t+r})\leq0$. Since $\e h(X_{t+r})$ must be positive on this set, we could have to compensate and $K_{t+r}$ is then positive for all $r\in[0,\m{R}'']$. Moreover, the compensation must be minimal i.e. such that $\e h(X_{t+r})=0$. Equation $\eqref{ito}$ becomes:
\begin{equation*}
\e\bigg(\int_t^{t+r} h'(X_{s^-}) dK_s\bigg)=-\int_t^{t+r}\e\m{L}h(X_{s^-}) ds,
\end{equation*}
on $[0,\m{R}'']$. Dividing both sides by $r$ and taking the limit $r\longrightarrow 0$ gives (by continuity):
\begin{equation*}
dK_t=-\frac{\e[\m{L}h(X_{t^-})]}{\e[h'(X_{t^-})]}dt.
\end{equation*}
Thus, $dK$ is absolutely continuous w.r.t. the Lebesgue measure with density:
\begin{equation*}
k_t=\frac{(\e[\m{L}h(X_{t^-})])^-}{\e[h'(X_{t^-})]}{\bf{1}}_{\e[h(X_t)]=0}.
\end{equation*}
\end{proof}

\begin{remark}
This justifies, at least under the smoothness Assumption $\ref{reg}$ on the constraint
function $h$, the non-negative hypothesis imposed on $h'$.
\end{remark}

\begin{proof}[Proof of Lemma \ref{lemmacont}]
Under Assumption \ref{bilip}, and by using Lemma \ref{lemmacont_generale}, we obtain the continuity of the function $t\longmapsto\e h(X_t)$.

Under the assumptions \ref{lip}, \ref{bilip} and \ref{reg}, we observe that $x\longmapsto\m{L}h(X_t)$ is a continuous function such that, for all $x\in\mathbb{R}$, there exists constants $C_1$, $C_2$ and $C_3 > 0$,
$$|b(x)h'(x)|\leq C_1(1+|x|),$$
$$|\sigma^2(x)h''(x)|\leq C_2(1+|x|^2),$$
and
\begin{equation*}
\begin{aligned}
\bigg|\int_E\bigg(h\big(x+F(x,z)\big)-h(x)-F(x,z)h'(x)\bigg)\lambda(dz)\bigg|&
\leq C_3\int_E|F(x,z)|\lambda(dz)\\&
\leq C_3\bigg(\int_E|F(x,z)-F(0,z)|\lambda(dz)+\int_E|F(0,z)|\lambda(dz)\bigg)\\&
\leq C_3\int_E|x|\lambda(dz)+C'_3\\&
\leq C_3(1+|x|).
\end{aligned}
\end{equation*}
Finally, by using Lemma \ref{lemmacont_generale}, we conclude the continuity of $t\longmapsto\e \m{L}h(X_t)$.
\end{proof}

\section{Mean reflected SDE as the limit of an interacting reflected particles system.}\label{sec:PMRSDE}

Having in mind the notations defined in the beginning of Section 2 and especially equation $\eqref{K_t}$, we can write the unique solution of the SDE \eqref{eq:main} as:
\begin{equation}
X_t  =X_0+\int_0^t b(X_{s^-}) ds + \int_0^t \sigma(X_{s^-}) dB_s  + \int_0^t\int_E F(X_{s^-},z) \tilde{N}(ds,dz)+\sup\limits_{s\le t} {G}_0(\mu_{s}),
\end{equation}
where $\mu_t$ stands for the law of
$$U_t  =X_0+\int_0^t b(X_{s^-}) ds + \int_0^t \sigma(X_{s^-}) dB_s  + \int_0^t\int_E F(X_{s^-},z) \tilde{N}(ds,dz).$$
We here are interested in the particle approximation of such a system. Our candidates are the particles, for $1\leq i\leq N$,
\begin{equation}
\label{particle_system}
X^i_t  = \bar{X}^i_0+\int_0^t b(X^i_{s^-}) ds + \int_0^t \sigma(X^i_{s^-}) dB^i_s  + \int_0^t\int_E F(X^i_{s^-},z) \tilde{N}^i(ds,dz)+\sup\limits_{s\le t} {G}_0(\mu_{s}^N),
\end{equation}
where $B^i$ are independent Brownian motions, $\tilde{N}^i$ are independent compensated Poisson measure, $\bar{X}^i_0$ are independent copies of $X_0$ and $\mu_s^N$ denotes the empirical distribution at time $s$ of the particles
\begin{equation*}
U^i_t  =\bar{X}^i_0+\int_0^t b(X^i_{s^-}) ds + \int_0^t \sigma(X^i_{s^-}) dB^i_s  + \int_0^t\int_E F(X^i_{s^-},z) \tilde{N}^i(ds,dz),\quad 1\leq i\leq N.
\end{equation*}
namely $\displaystyle\mu^N_s=\frac{1}{N}\sum^N_{i=1}\delta_{U^i_s}$.
It is worth noticing that
\begin{equation*}
G_0(\mu_s^N)=\inf\Bigg\{x\geq0:\frac{1}{N}\sum_{i=1}^{N}h(x+U^i_s) \geq 0\Bigg\}.
\end{equation*}

In order to prove that there is indeed a propagation of chaos effect, we introduce the following independent copies of $X$
\begin{equation*}
\bar X^i_t  =\bar X^i_0+\int_0^t b(\bar X^i_{s^-}) ds + \int_0^t \sigma(\bar X^i_{s^-}) dB^i_s  + \int_0^t\int_E F(\bar X^i_{s^-},z) \tilde{N}^i(ds,dz)+\sup\limits_{s\le t} {G}_0(\mu_{s}),\quad 1\leq i\leq N,
\end{equation*}
and we couple these particles with the previous ones by choosing the same Brownian motions and the same Poisson processes.\\
In order to do so, we introduce the decoupled particles $\bar U^i$, $1\leq i\leq N$: 
\begin{equation*}
\bar U^i_t  =\bar X^i_0+\int_0^t b(\bar X^i_{s^-}) ds + \int_0^t \sigma(\bar X^i_{s^-}) dB^i_s  + \int_0^t\int_E F(\bar X^i_{s^-},z) \tilde{N}^i(ds,dz).
\end{equation*}
Note that for instance the particles $\bar U^i_t$ are i.i.d. and let us denote by $\bar\mu^N$ the empirical measure associated to this system of particles.\\

\begin{remark}
\begin{itemize}
\item[(i)] Under our assumptions, we have $\e\left[h\left(\bar{X}^i_0\right)\right]=\e\left[h(X_0)\right]\geq 0$. However, there is no reason to have 
\begin{equation*}
	\frac{1}{N}\, \sum_{i=1}^N h\left(\bar{X}^i_0\right) \geq 0,
\end{equation*}
even if $N$ is large. As a consequence,
\begin{equation*}
	{G}_0(\mu_{0}^N)=\inf\left\{x\geq 0:\frac{1}{N}\sum_{i=1}^{N}h\left(x+\bar{X}^i_0\right) \geq 0\right\}
\end{equation*}
is not necessarily equal to 0. As a byproduct, we have $X^i_0 = \bar{X}^i_0 + {G}_0(\mu_{0}^N)$ and the non decreasing process $\sup_{s\le t} {G}_0(\mu_{s}^N)$ is not equal to 0 at time $t=0$. Written in this way, the particles defined by \eqref{particle_system} can not be interpreted as the solution of a reflected SDE. To view the particles as the solution of a reflected SDE, instead of \eqref{particle_system} one has to solve
\begin{gather*}
	X^i_t  = \bar{X}^i_0 + {G}_0(\mu_{0}^N) +\int_0^t b(X^i_{s^-}) ds + \int_0^t \sigma(X^i_{s^-}) dB^i_s  + \int_0^t\int_E F(X^i_{s^-},z) \tilde{N}^i(ds,dz)+ K^N_t, \\
	\frac{1}{N}\sum_{i=1}^{N}h\left(X^i_t\right) \geq 0, \qquad \int_0^t \frac{1}{N}\sum_{i=1}^{N}h\left(X^i_t\right)\, dK^N_s,
\end{gather*}
with $K^N$ non decreasing and $K^N_0=0$. Since, we do not use this point in the sequel, we will work with the form \eqref{particle_system}.

\item[(ii)] Following the same proof of Theorem $\ref{thrm_exacte}$, it is easy to demonstrate the existence and the uniqueness of a solution for the particle approximated system $\eqref{particle_system}$.
\end{itemize}
\end{remark}

We have the following result concerning the approximation \eqref{eq:main} by interacting particles system.

\begin{theorem}
\label{cv_part}
Let Assumptions \ref{lip} and \ref{bilip} hold and $T>0$.
\begin{enumerate}
\item [(i)] Under Assumption \ref{int}, there exists a constant $C$ depending on $b$, $\sigma$ and $F$ such that, for each $j\in\{1,\ldots,N\}$,
\begin{equation*}
\e\bigg[\sup_{s\leq T}|X^j_s-\bar X^j_s|^2\bigg]\leq C\exp\bigg(C\bigg(1+\frac{M^2}{m^2}\bigg)(1+T^2)\bigg)\frac{M^2}{m^2}N^{-1/2}.
\end{equation*}

\item [(ii)] If Assumption \ref{reg} is in force, then there exists a constant $C$ depending on $b$, $\sigma$ and $F$ such that, for each $j\in\{1,\ldots,N\}$,
\begin{equation*}
\e\bigg[\sup_{s\leq T}|X^j_s-\bar X^j_s|^2\bigg]\leq C\exp\bigg(C\bigg(1+\frac{M^2}{m^2}\bigg)(1+T^2)\bigg)\frac{1+T^2}{m^2}\bigg(1+\e\bigg[\sup_{s\leq T}|X_T|^2\bigg]\bigg)N^{-1}.
\end{equation*}
\end{enumerate}
\end{theorem}

\begin{proof}
Let $t>0$. We have, for $r\leq t$,
\begin{equation*}
\begin{aligned}
\big|X^j_r-\bar{X}^j_r\big|& \leq \bigg|\int_0^r b(X^j_{s^-})- b(\bar{X}^j_{{s}^-})ds\bigg| + \bigg|\int_0^r \bigg(\sigma(X^j_{s^-})-\sigma(\bar{X}^j_{{s}^-})\bigg) dB^i_s\bigg|  \\&~~
~~+ \bigg|\int_0^r\int_E\bigg(F(X^j_{s^-},z) -F(\bar{X}^j_{{s}^-},z)\bigg) \tilde{N}^j(ds,dz)\bigg| +  \big|\sup_{s\le r}{G}_0(\mu_s^N)-\sup_{s\le r}{G}_0({\mu}_{s})\big|.
\end{aligned}
\end{equation*}
Taking into account the fact that
\begin{equation*}
\begin{aligned}
\big|\sup_{s\le r}{G}_0(\mu_s^N)-\sup_{s\le r}{G}_0({\mu}_{s})\big|&
\leq \sup_{s\le r}\big|{G}_0(\mu_s^N)-{G}_0({\mu}_{s})\big| 
\leq \sup_{s\le t}\big|{G}_0(\mu_s^N)-{G}_0({\mu}_{s})\big|\\& 
\leq \sup_{s\le t}\big|{G}_0(\mu_s^N)-{G}_0({\bar\mu}^N_{s})\big|+\sup_{s\le t}\big|{G}_0(\bar\mu_s^N)-{G}_0({\mu}_{s})\big|,
\end{aligned}
\end{equation*}
we get the inequality
\begin{equation}
\label{inq1_proof_part}
\begin{aligned}
\sup_{r\le t}\big|X^j_r-\bar{X}^j_r\big|& \leq I_1+\sup_{s\le t}\big|{G}_0(\mu_s^N)-{G}_0({\bar\mu}^N_{s})\big|+\sup_{s\le t}\big|{G}_0(\bar\mu_s^N)-{G}_0({\mu}_{s})\big|,
\end{aligned}
\end{equation}
where we have set
\begin{equation*}
\begin{aligned}
I_1&=\int_0^t \big|b(X^j_{s^-})- b(\bar{X}^j_{{s}^-})\big|ds + \sup_{r\le t}\bigg|\int_0^r \bigg(\sigma(X^j_{s^-})-\sigma(\bar{X}^j_{{s}^-})\bigg) dB^i_s\bigg|\\&~~~~+ \sup_{r\le t}\bigg|\int_0^r\int_E\bigg(F(X^j_{s^-},z) -F(\bar{X}^j_{{s}^-},z)\bigg) \tilde{N}^j(ds,dz)\bigg|.
\end{aligned}
\end{equation*}

On the one hand we have, using Assumption \ref{lip}, Doob and Cauchy-Schwartz inequalities
\begin{equation*}
\begin{aligned}
\e\big[\big|I_1\big|^2\big]&
\leq C\Bigg\{\e \bigg[t\int_0^t\bigg|b(X^j_{s^-})- b(\bar{X}^j_{{s}^-}) \bigg|^2ds\bigg] + \e\bigg[\int_0^t \bigg|\sigma(X^j_{s^-})-\sigma(\bar{X}^j_{{s}^-})\bigg|^2 ds\bigg]\\&~~
~~+\e \bigg[\int_0^t\int_E\bigg|F(X^j_{s^-},z) -F(\bar{X}^j_{{s}^-},z)\bigg|^2 \lambda(dz)ds\bigg]\Bigg\} \\&
\leq C\Bigg\{tC_1\int_0^t \e\bigg[\big|X^j_s-\bar{X}^j_{{s}}\big|^2\bigg] ds+ C_1\int_0^t \e\bigg[\big|X^j_s-\bar{X}^j_{{s}}\big|^2\bigg] ds\\&~~
~~+C_1\int_0^t \e\bigg[\big|X^j_s-\bar{X}^j_{{s}}\big|^2\bigg] ds\Bigg\}\\&
\leq C(1+t)\int_0^t \e\bigg[\big|X^j_s-\bar{X}^j_{{s}}\big|^2\bigg] ds.
\end{aligned}
\end{equation*}
where $C$ depends only on $b$, $\sigma$ and $F$ and may change from line to line.

On the other hand, by using Lemma \ref{wiss},
\begin{equation*}
\sup_{s\le t}\big|{G}_0(\mu_s^N)-{G}_0({\bar\mu}^N_{s})\big|
\leq \frac{M}{m}\sup_{s\le t}\frac{1}{N}\sum_{i=1}^{N}\big|U^i_s-\bar{U}^i_{{s}}\big|
\leq \frac{M}{m}\frac{1}{N}\sum_{i=1}^{N}\sup_{s\le t}\big|U^i_s-\bar{U}^i_{{s}}\big|,
\end{equation*}
and Cauchy-Schwartz inequality gives, since the variables are exchangeable,
\begin{equation*}
\e\bigg[\sup_{s\le t}\big|{G}_0(\mu_s^N)-{G}_0({\bar\mu}^N_{s})\big|^2\bigg]
\leq \frac{M^2}{m^2}\frac{1}{N}\sum_{i=1}^{N}\e\bigg[\sup_{s\le t}\big|U^i_s-\bar{U}^i_{{s}}\big|^2\bigg]
= \frac{M^2}{m^2}\e\bigg[\sup_{s\le t}\big|U^j_s-\bar{U}^j_{{s}}\big|^2\bigg].
\end{equation*}
Since
\begin{equation*}
U^j_s-\bar{U}^j_{{s}}=\int_0^s (b(X^j_{r^-})-b(\bar X^j_{r^-})) dr + \int_0^s (\sigma(X^j_{r^-})-\sigma(\bar X^j_{r^-})) dB^j_r  + \int_0^s\int_E (F(X^j_{r^-},z)-F(\bar X^j_{r^-}),z) \tilde{N}^j(dr,dz),
\end{equation*}
the same computations as we did before lead to
\begin{equation*}
\e\bigg[\sup_{s\le t}\big|{G}_0(\mu_s^N)-{G}_0({\bar\mu}^N_{s})\big|^2\bigg]
\leq C \frac{M^2}{m^2}(1+t)\int_0^t \e\bigg[\big|X^j_s-\bar{X}^j_{{s}}\big|^2\bigg] ds.
\end{equation*}
Hence, with the previous estimates we get, coming back to $\eqref{inq1_proof_part}$,
\begin{equation*}
\begin{aligned}
\e\bigg[\sup_{r\le t}\big|X^j_r-\bar{X}^j_r\big|^2\bigg]&
\leq K\int_0^t \e\bigg[\big|X^j_s-\bar{X}^j_{{s}}\big|^2\bigg] ds+4\e\bigg[\sup_{s\le t}\big|{G}_0(\bar\mu_s^N)-{G}_0({\mu}_{s})\big|^2\bigg]\\&
\leq K\int_0^t \e\bigg[\sup_{r\le s}\big|X^j_r-\bar{X}^j_{{r}}\big|^2\bigg] ds+4\e\bigg[\sup_{s\le t}\big|{G}_0(\bar\mu_s^N)-{G}_0({\mu}_{s})\big|^2\bigg],
\end{aligned}
\end{equation*}
where $K=C(1+t)(1+M^2/m^2)$. Thanks to Gronwall's Lemma
\begin{equation*}
\begin{aligned}
\e\bigg[\sup_{r\le t}\big|X^j_r-\bar{X}^j_r\big|^2\bigg]&
\leq Ce^{Kt}\e\bigg[\sup_{s\le t}\big|{G}_0(\bar\mu_s^N)-{G}_0({\mu}_{s})\big|^2\bigg].
\end{aligned}
\end{equation*}

By Lemma \ref{wiss} we know that
\begin{equation*}
\begin{aligned}
\e\bigg[\sup_{s\le t}\big|{G}_0(\bar\mu_s^N)-{G}_0({\mu}_{s})\big|^2\bigg]
\leq {1\over{m^2}}\e\bigg[\sup_{s\le t}\left|\int h(\bar{G}_0(\mu_s)+\cdot) (d\bar\mu_s^N-d\mu_s)\right|^2\bigg],
\end{aligned}
\end{equation*}
from which we deduce that
\begin{equation}
\label{inq2_proof_part}
\begin{aligned}
\e\bigg[\sup_{r\le t}\big|X^j_r-\bar{X}^j_r\big|^2\bigg]&
\leq Ce^{Kt}{1\over{m^2}}\e\bigg[\sup_{s\le t}\left|\int h(\bar{G}_0(\mu_s)+\cdot) (d\bar\mu_s^N-d\mu_s)\right|^2\bigg].
\end{aligned}
\end{equation}
\begin{proof}[Proof of (i)]\renewcommand{\qedsymbol}{}
Since the function $h$ is, at least, a Lipschitz function, we understand that the rate of convergence follows from the convergence of empirical measure of i.i.d. diffusion processes.
The crucial point here is that we consider a uniform (in time) convergence, which may possibly damage the usual rate of convergence. We will see that however here, we are able to conserve
this optimal rate. Indeed, in full generality (i.e. if we only suppose that \ref{bilip} holds) we get that:
\begin{equation*}
\begin{aligned}
{1\over{m^2}}\e\bigg[\sup_{s\le t}\left|\int h(\bar{G}_0(\mu_s)+\cdot) (d\bar\mu_s^N-d\mu_s)\right|^2\bigg]
\leq \frac{M^2}{m^2}\e\bigg[\sup_{s\le t}W_1^2(\bar{\mu}_s^N,\mu_s)\bigg].
\end{aligned}
\end{equation*}
Thanks to the additional Assumption \ref{int} and to Remark \ref{rem_U_1} and Proposition \ref{propriete_2}, we will adapt and simplify the proof of Theorem 10.2.7 of \cite{RR98} using recent results about the control Wasserstein distance of empirical measures of i.i.d. sample to the true law by \cite{FG15}, to obtain
\begin{equation*}
\e\bigg[\sup_{s\le 1}W_1^2(\bar{\mu}_s^N,\mu_s)\bigg] \leq CN^{-1/2}.
\end{equation*}
The reader can refer to the work on (\cite{BCGL17}, Theorem 3.2, Proof of (i)) in order to find this result.
\end{proof}

\begin{proof}[Proof of (ii)]\renewcommand{\qedsymbol}{}
In the case where $h$ is a twice continuously differentiable function with bounded derivatives (i.e. under \ref{reg}), we succeed to take benefit from the fact that $\bar\mu^N$ is an empirical measure associated to i.i.d. copies of diffusion process, in particular we can get rid of the supremum in time. In view of $\eqref{inq2_proof_part}$, we need a sharp estimate of
\begin{equation}
\label{more order}
\e\bigg[\sup_{s\le t}\left|\int h(\bar{G}_0(\mu_s)+\cdot) (d\bar\mu_s^N-d\mu_s)\right|^2\bigg].
\end{equation}
Let us first observe that $s\longmapsto\bar{G}_0(\mu_{s})$ is locally Lipschitz continuous. Indeed, since by definition $H(\bar{G}_0(\mu_{t}),\mu_t)=0$, if $s<t$, using $\eqref{Hbilip}$,

\begin{equation*}
\begin{aligned}
|\bar{G}_0(\mu_{s})-\bar{G}_0(\mu_{t})|&\leq \frac{1}{m}|H(\bar{G}_0(\mu_{s}),\mu_t)-H(\bar{G}_0(\mu_{t}),\mu_t)|\\&=\frac{1}{m}|H(\bar{G}_0(\mu_{s}),\mu_t)|,\\&=\frac{1}{m}|\e[h(\bar{G}_0(\mu_{s})+U_t)]|\\&
=\frac{1}{m}\bigg|\e\bigg[h\bigg(\bar{G}_0(\mu_{s})+U_s+\int_{s}^{t}b(X_{r^-})dr+\int_{s}^{t}\sigma(X_{r^-})dB_r+\int_{s}^{t}\int_E F(X_{r^-},z)\tilde{N}(dr,dz)\bigg)\bigg]\bigg|.
\end{aligned}
\end{equation*}

We get from It\^o's formula, setting
\begin{equation*}
\bar{\m{L}}_yf(x):=b(y)\frac{\partial}{\partial x}f(x)+\frac{1}{2}\sigma\sigma^*(y)\frac{\partial^2}{\partial x^2}f(x)+\int_E\bigg(f\big(x+F(y,z)\big)-f(x)-F(y,z)f'(x)\bigg)\lambda(dz),
\end{equation*}
\begin{equation*}
\begin{aligned}
h(\bar{G}_0(\mu_{s})+U_t)&=h(\bar{G}_0(\mu_{s})+U_s)+\int_s^tb\big(X_{r^-}\big)h'\big(\bar{G}_0(\mu_{s})+U_{r^-}\big)dr+\int_s^t\sigma\big(X_{r^-}\big)h'\big(\bar{G}_0(\mu_{s})+U_{r^-}\big)dB_r\\&~~~~+\int_s^t\int_E F\big(X_{r^-},z\big)h'\big(\bar{G}_0(\mu_{s})+U_{r^-}\big)\tilde{N}(dr,dz)+\frac{1}{2}\int_s^t\sigma^2\big(X_{r^-}\big)h''\big(\bar{G}_0(\mu_{s})+U_{r^-}\big)dr\\&~~~~+\int_s^t\int_E m(r,z)\lambda(dz)dr+\int_s^t\int_E m(r,z) \tilde{N}(dr,dz),
\end{aligned}
\end{equation*}
with
$$m(r,z)=\bigg(h\big(\bar{G}_0(\mu_{s})+U_{r^-}+F(X_{r^-},z)\big)-h\big(\bar{G}_0(\mu_{s})+U_{r^-}\big)-F\big(X_{r^-},z\big)h'\big(\bar{G}_0(\mu_{s})+U_{r^-}\big)\bigg).$$
Thus, we obtain
\begin{equation*}
\begin{aligned}
h(\bar{G}_0(\mu_{s})+U_t)&
=h(\bar{G}_0(\mu_{s})+U_s)+\int_s^t\bar{\m{L}}_{X_{r^-}}h(\bar{G}_0(\mu_{s})+U_{r^-})dr+\int_s^t\sigma(X_{r^-})h'(\bar{G}_0(\mu_{s})+U_{r^-})dB_r\\&~~~~+\int_s^t\int_E\bigg(h\big(\bar{G}_0(\mu_{s})+U_{r^-}+F(X_{r^-},z)\big)-h\big(\bar{G}_0(\mu_{s})+U_{r^-}\big)\bigg)\tilde{N}(dr,dz),
\end{aligned}
\end{equation*}
and so
\begin{equation*}
\begin{aligned}
\e[h(\bar{G}_0(\mu_{s})+U_t)]&=\e[h(\bar{G}_0(\mu_{s})+U_s)]+\int_s^t\e[\bar{\m{L}}_{X_{r^-}}h(\bar{G}_0(\mu_{s})+U_{r^-})]dr\\&
=H(\bar{G}_0(\mu_{s}),\mu_s)+\int_s^t\e[\bar{\m{L}}_{X_{r^-}}h(\bar{G}_0(\mu_{s})+U_{r^-})]dr\\&
=\int_s^t\e[\bar{\m{L}}_{X_{r^-}}h(\bar{G}_0(\mu_{s})+U_{r^-})]dr.
\end{aligned}
\end{equation*}
Since $h$ has bounded derivatives and $\sup_{s\leq T}|X_s|$ is a square integrable random variable for each $T > 0$ (see Proposition \ref{propriete_1}), the result follows easily.

Now, we deal with \eqref{more order}.\\
Let us denote by $\psi$ the Radon-Nikodym derivative of $\bar{G}_0(\mu)$. By definition, we have, denoting by $V^i$ the semi-martingale $s\longmapsto \bar{G}_0(\mu_s)+\bar{U}_s^i$, since $\bar{U}^i$ are independent copies of $U$,
\begin{equation*}
\begin{aligned}
R_N(s):=\int h(\bar{G}_0(\mu_s)+\cdot) (d\bar\mu_s^N-d\mu_s)&
=\frac{1}{N}\sum_{i=1}^N h\big(\bar{G}_0(\mu_s)+\bar{U}_s^i\big)-\e\big[h\big(\bar{G}_0(\mu_s)+{U}_s\big)\big]\\&
=\frac{1}{N}\sum_{i=1}^N \big\{h\big(\bar{G}_0(\mu_s)+\bar{U}_s^i\big)-\e\big[h\big(\bar{G}_0(\mu_s)+\bar{U}_s^i\big)\big]\big\}\\&
=\frac{1}{N}\sum_{i=1}^N \big\{h\big(V^i_s\big)-\e\big[h\big(V_s^i\big)\big]\big\},
\end{aligned}
\end{equation*}
It follows from It\^o's formula
\begin{equation*}
\begin{aligned}
h\big(V^i_s\big)&=h\big(V^i_0\big)+\int_{0}^{s}h'\big(V^i_{r^-}\big)d\bar{G}_0\big(\mu_r\big)+\int_{0}^{s}b\big(\bar{X}^i_{r^-}\big)h'\big(V^i_{r^-}\big)dr+\int_{0}^{s}\sigma\big(\bar{X}^i_{r^-}\big)h'\big(V^i_{r^-}\big)dB^i_r\\&~~~~+\int_{0}^{s}\int_EF\big(\bar{X}^i_{r^-},z\big)h'\big(V^i_{r^-}\big)\tilde{N}^i(dr,dz)+\frac{1}{2}\int_{0}^{s}\sigma^2\big(\bar{X}^i_{r^-}\big)h''\big(V^i_{r^-}\big)dr\\&~~~~+\int_0^{s}\int_E\bigg(h\big(V^i_{r^-}+F(\bar{X}^i_{r^-},z)\big)-h(V^i_{r^-})-F(\bar{X}^i_{r^-},z)h'(V^i_{r^-})\bigg)\lambda(dz)dr\\&~~~~+\int_0^s\int_E\bigg(h\big(V^i_{r^-}+F(\bar{X}^i_{r^-},z)\big)-h(V^i_{r^-})-F(\bar{X}^i_{r^-},z)h'(V^i_{r^-})\bigg)\tilde{N}^i(dr,dz)\\&
=h\big(V^i_0\big)+\int_{0}^{s}h'\big(V^i_{r^-}\big)\psi_r dr+\int_{0}^{s}b\big(\bar{X}^i_{r^-}\big)h'\big(V^i_{r^-}\big)dr+\frac{1}{2}\int_{0}^{s}\sigma^2\big(\bar{X}^i_{r^-}\big)h''\big(V^i_{r^-}\big)dr\\&~~~~+\int_0^{s}\int_E\bigg(h\big(V^i_{r^-}+F(\bar{X}^i_{r^-},z)\big)-h(V^i_{r^-})-F(\bar{X}^i_{r^-},z)h'(V^i_{r^-})\bigg)\lambda(dz)dr\\&~~~~+\int_{0}^{s}\sigma\big(\bar{X}^i_{r^-}\big)h'\big(V^i_{r^-}\big)dB^i_r+\int_0^s\int_E\bigg(h\big(V^i_{r^-}+F(\bar{X}^i_{r^-},z)\big)-h(V^i_{r^-})\bigg)\tilde{N}^i(dr,dz)\\&
=h\big(V^i_0\big)+\int_{0}^{s}h'\big(V^i_{r^-}\big)\psi_r dr+\int_{0}^{s}\bar{\m{L}}_{\bar{X}^i_{r^-}}h\big(V^i_{r^-}\big)dr+\int_{0}^{s}h'\big(V^i_{r^-}\big)\sigma\big(\bar{X}^i_{r^-}\big)dB^i_r\\&~~~~+\int_{0}^{s}\int_E\bigg(h\big(V^i_{r^-}+F(\bar{X}^i_{r^-},z)\big)-h\big(V^i_{r^-}\big)\bigg)\tilde{N}^i(dr,dz)\\&
=h\big(V^i_0\big)+\int_{0}^{s}\big\{h'\big(V^i_{r^-}\big)\psi_r +\bar{\m{L}}_{\bar{X}^i_{r^-}}h\big(V^i_{r^-}\big)\big\}dr+\int_{0}^{s}h'\big(V^i_{r^-}\big)\sigma\big(\bar{X}^i_{r^-}\big)dB^i_r\\&~~~~+\int_{0}^{s}\int_E\bigg(h\big(V^i_{r^-}+F(\bar{X}^i_{r^-},z)\big)-h\big(V^i_{r^-}\big)\bigg)\tilde{N}^i(dr,dz),
\end{aligned}
\end{equation*}
Taking expectation gives
\begin{equation*}
\begin{aligned}
\e\bigg[h\big(V^i_s\big)\bigg]&
=\e\bigg[h\big(V^i_0\big)\bigg]+\int_{0}^{s}\e\big[h'\big(V^i_{r^-}\big)\psi_r +\bar{\m{L}}_{\bar{X}^i_{r^-}}h\big(V^i_{r^-}\big)\big]dr\\&
=H(\bar{G}_0(\mu_{0}),\mu_0)+\int_{0}^{s}\e\big[h'\big(V^i_{r^-}\big)\psi_r +\bar{\m{L}}_{\bar{X}^i_{r^-}}h\big(V^i_{r^-}\big)\big]dr\\&
=0+\int_{0}^{s}\e\big[h'\big(V^i_{r^-}\big)\psi_r +\bar{\m{L}}_{\bar{X}^i_{r^-}}h\big(V^i_{r^-}\big)\big]dr.
\end{aligned}
\end{equation*}
We deduce immediately that
\begin{equation*}
\begin{aligned}
R_N(s)&=\frac{1}{N}\sum_{i=1}^N h(V^i_0)+\frac{1}{N}\sum_{i=1}^N \int_{0}^{s}C^i(r) dr +M_N(s)+L_N(s)\\&
=\frac{1}{N}\sum_{i=1}^N h(V^i_0)+ \int_{0}^{s}\bigg(\frac{1}{N}\sum_{i=1}^N C^i(r)\bigg) dr +M_N(s)+L_N(s),
\end{aligned}
\end{equation*}
where we have set 
$$C^i(r)=h'\big(V^i_{r^-}\big)\psi_r +\bar{\m{L}}_{\bar{X}^i_{r^-}}h\big(V^i_{r^-}\big)-\e\big[h'\big(V^i_{r^-}\big)\psi_r +\bar{\m{L}}_{\bar{X}^i_{r^-}}h\big(V^i_{r^-}\big)\big],$$
$$M_N(s)=\frac{1}{N}\sum_{i=1}^N\int_{0}^{s}h'\big(V^i_{r^-}\big)\sigma\big(\bar{X}^i_{r^-}\big)dB^i_r,$$
$$L_N(s)=\frac{1}{N}\sum_{i=1}^N\int_{0}^{s}\int_E\Big(h\big(V^i_{r^-}+F(\bar{X}^i_{r^-},z)\big)-h\big(V^i_{r^-}\big)\Big)\tilde{N}^i(dr,dz).$$

As a byproduct,
\begin{equation*}
\begin{aligned}
\sup_{s\le t}|R_N(s)|&
\leq\bigg|\frac{1}{N}\sum_{i=1}^N h(V^i_0)\bigg|+ \sup_{s\le t}\int_{0}^{s}\bigg|\frac{1}{N}\sum_{i=1}^N C^i(r)\bigg| dr +\sup_{s\le t}|M_N(s)|+\sup_{s\le t}|L_N(s)|\\&
\leq \bigg|\frac{1}{N}\sum_{i=1}^N h(V^i_0)\bigg|+\int_{0}^{t}\bigg|\frac{1}{N}\sum_{i=1}^N C^i(r)\bigg| dr +\sup_{s\le t}|M_N(s)|+\sup_{s\le t}|L_N(s)|.
\end{aligned}
\end{equation*}

We get, using Cauchy-Schwartz inequality, since $U^i$ and $\bar{X}^i$ are i.i.d,
\begin{equation*}
\begin{aligned}
\e\bigg[\sup_{s\le t}|R_N(s)|^2\bigg]&
\leq 4\Bigg\{\mathbb{V}\bigg[\frac{1}{N}\sum_{i=1}^N h(V^i_0)\bigg]+\e\bigg[\bigg(\int_{0}^{t}\bigg|\frac{1}{N}\sum_{i=1}^N C^i(r)\bigg| dr\bigg)^2\bigg] \\&~~~~+\e\bigg[\sup_{s\le t}|M_N(s)|^2\bigg]+\e\bigg[\sup_{s\le t}|L_N(s)|^2\bigg]\Bigg\}\\&
\leq 4\Bigg\{\mathbb{V}\bigg[\frac{1}{N}\sum_{i=1}^N h(V^i_0)\bigg]+t\e\bigg[\int_{0}^{t}\bigg|\frac{1}{N}\sum_{i=1}^N C^i(r)\bigg|^2 dr\bigg] \\&~~~~+\e\bigg[\sup_{s\le t}|M_N(s)|^2\bigg]+\e\bigg[\sup_{s\le t}|L_N(s)|^2\bigg]\Bigg\}\\&
= 4\Bigg\{\mathbb{V}\bigg[\frac{1}{N}\sum_{i=1}^N h(V^i_0)\bigg]+t\int_{0}^{t}\mathbb{V}\bigg(\frac{1}{N}\sum_{i=1}^N C^i(r)\bigg) dr \\&~~~~+\e\bigg[\sup_{s\le t}|M_N(s)|^2\bigg]+\e\bigg[\sup_{s\le t}|L_N(s)|^2\bigg]\Bigg\}.
\end{aligned}
\end{equation*} 
Thus, we get
\begin{equation*}
\begin{aligned}
\e\bigg[\sup_{s\le t}|R_N(s)|^2\bigg]&
\leq \frac{4}{N} \mathbb{V}[h(V_0)]+\frac{4t}{N}\int_{0}^{t}\mathbb{V}(C(r)) dr +4\e\bigg[\sup_{s\le t}|M_N(s)|^2\bigg]+4\e\bigg[\sup_{s\le t}|L_N(s)|^2\bigg]\\&
=\frac{4}{N} \mathbb{V}[h(V_0)]+\frac{4t}{N}\int_{0}^{t}\mathbb{V}(h'\big(V_{r^-}\big)\psi_r +\bar{\m{L}}_{{X}_{r^-}}h\big(V_{r^-}\big)) dr \\&~~~~+4\e\bigg[\sup_{s\le t}|M_N(s)|^2\bigg]+4\e\bigg[\sup_{s\le t}|L_N(s)|^2\bigg].
\end{aligned}
\end{equation*} 
Since $M_N$ is a martingale with
$$\langle M_N\rangle_t=\frac{1}{N^2}\sum_{i=1}^N\int_{0}^t\big(h'(V^i_{r^-})\sigma(\bar{X}^i_{r^-})\big)^2dr,$$
Doob's inequality leads to
\begin{equation*}
\begin{aligned}
\e\bigg[\sup_{s\le t}|M_N(s)|^2\bigg]&\leq 4\e\big[|M_N(t)|^2\big]\\&=\frac{4}{N^2}\sum_{i=1}^N\int_{0}^t\e\bigg[\big(h'(V^i_{r^-})\sigma(\bar{X}^i_{r^-})\big)^2\bigg]dr\\&=\frac{4}{N}\int_{0}^t\e\bigg[\big(h'(V_{r^-})\sigma({X}_{r^-})\big)^2\bigg]dr.
\end{aligned}
\end{equation*}
Then, using Doob inequality for the last martingale $L_N$, we obtain
\begin{equation*}
\begin{aligned}
\e\bigg[\sup_{s\le t}|L_N(s)|^2\bigg]&\leq 4\e\big[|L_N(t)|^2\big]\\&=\frac{4}{N^2}\sum_{i=1}^N\e\bigg[\bigg(\int_{0}^t\int_E\Big(h\big(V^i_{r^-}+F(\bar{X}^i_{r^-},z)\big)-h\big(V^i_{r^-}\big)\Big)\tilde{N}^i(dr,dz)\bigg)^2\bigg]\\&~~~~+\frac{8}{N^2}\sum_{1\leq i<j\leq N}\e\bigg[\int_{0}^t\int_E\Big(h\big(V^i_{r^-}+F(\bar{X}^i_{r^-},z)\big)-h\big(V^i_{r^-}\big)\Big)\tilde{N}^i(dr,dz)\\&~~~~\int_{0}^t\int_E\Big(h\big(V^j_{r^-}+F(\bar{X}^j_{r^-},z)\big)-h\big(V^j_{r^-}\big)\Big)\tilde{N}^j(dr,dz)\bigg]\\&=\frac{4}{N^2}\sum_{i=1}^N\int_{0}^t\int_E\e\bigg[\Big(h\big(V^i_{r^-}+F(\bar{X}^i_{r^-},z)\big)-h\big(V^i_{r^-}\big)\Big)^2\bigg]\lambda(dz)dr+0\\&=\frac{4}{N}\int_{0}^t\int_E\e\bigg[\Big(h\big(V^i_{r^-}+F(\bar{X}^i_{r^-},z)\big)-h\big(V^i_{r^-}\big)\Big)^2\bigg]\lambda(dz)dr.
\end{aligned}
\end{equation*}
Finally, using the fact that $h$ has bounded derivatives, $b$, $\sigma$ and $F$ are Lipschitz, we get
$$\e\bigg[\sup_{s\le t}|R_N(s)|^2\bigg]\leq C(1+t^2)\bigg(1+\e\bigg[\sup_{s\le t}|X_s|^2\bigg]\bigg)N^{-1}.$$
This gives the result coming back to $\eqref{inq2_proof_part}$.
\end{proof}
\end{proof}

\section{A numerical scheme for MRSDE.}\label{sec:NSMRSDE}

We are interested in the numerical approximation of the SDE \eqref{eq:main} on $[0,T]$. Here are the main steps of the scheme. Let $0=T_0<T_1<\cdots<T_n=T$ be a subdivision of $[0,T]$. Given this subdivision, we denote by "$\_$" the mapping $s\mapsto\underline{s}=T_k$ if $s\in[T_k,T_{k+1})$, $k\in \{0,\cdots ,n-1\}$. For simplicity, we consider only the case of regular subdivisions: for a given integer $n$, $T_k=kT/n$, $k=0,\ldots,n$.\\
Let us recall that we proved in the previous section that the particles system, for $1\leq i\leq N$,
$$
X^i_t=\bar{X}^i_0+\int_0^t b(X^i_{s^-}) ds + \int_0^t \sigma(X^i_{s^-}) dB^i_s  + \int_0^t\int_E F(X^i_{s^-},z) \tilde{N}^i(ds,dz) + \sup_{s\le t} {G}_0(\mu^N_{s}),$$
where we have set 
\begin{align*}
	\mu^N_t & =\frac{1}{N}\sum^N_{i=1}\delta_{U^i_t}, \\
	U^i_t & =\bar{X}^i_0+\int_0^t b(X^i_{s^-}) ds + \int_0^t \sigma(X^i_{s^-}) dB^i_s  + \int_0^t\int_E F(X^i_{s^-},z) \tilde{N}^i(ds,dz),\quad 1 \leq i\leq N, 
\end{align*}
$B^i$ being independent Brownian motions, $N^i$ being independent Poisson processes and $\bar{X}^i_0$ being independent copies of $X_0$, converges toward the solution to \eqref{eq:main}. Thus, the numerical approximation is obtained by an Euler scheme applied to this particles system. We introduce the following discrete version of the particles system: for $1\leq i\leq N$,
$$
\tilde{X}^i_t=\bar{X}^i_0+\int_0^t b(\tilde{X}^i_{\underline{s}^-}) ds + \int_0^t \sigma(\tilde{X}^i_{\underline{s}^-}) dB^i_s  + \int_0^t\int_E F(\tilde{X}^i_{\underline{s}^-},z) \tilde{N}^i(ds,dz) + \sup_{s\le t} {G}_0(\tilde{\mu}^N_{\underline{s}}),
$$
with the notation 
\begin{align*}
	\tilde{\mu}^N_{\underline{t}} & =\frac{1}{N}\sum^N_{i=1}\delta_{\tilde{U}^i_t}, \\
	\tilde{U}^i_t & =\bar{X}^i_0+\int_0^t b(\tilde{X}^i_{\underline{s}^-}) ds + \int_0^t \sigma(\tilde{X}^i_{\underline{s}^-}) dB^i_s  + \int_0^t\int_E F(\tilde{X}^i_{\underline{s}^-},z) \tilde{N}^i(ds,dz),\quad 1 \leq i\leq N.
\end{align*}

\subsection{Scheme.}
Using the notations given above, the result on the interacting system of mean
reflected particles of the MR-SDE of Section 3 and Remark $\ref{rem}$, we deduce the following algorithm for the numerical approximation of the MR-SDE:

\begin{remark}
We emphasize that, at each step $k$ of the algorithm, we approximate the increment of the reflection process $K$ by the increment of the approximation:
\begin{equation}
\label{increment} 
\Delta_k \hat{K}^N:=\sup_{l\leq k}G_0(\tilde{\mu}^N_{T_l})-\sup_{l\leq k-1}G_0(\tilde{\mu}^N_{T_l}).
\end{equation}
\end{remark}

First, we consider the special case when the SDE is defined by
\begin{equation*}
	\begin{cases}
	\begin{split}
	& X_t  =X_0+\int_0^t b(X_{s^-}) ds + \int_0^t \sigma(X_{s^-}) dB_s  + \int_0^t F(X_{s^-}) d\tilde{N}_s + K_t,\quad t\geq 0, \\
	& \e[h(X_t)] \geq 0, \quad \int_0^t \e[h(X_s)] \, dK_s = 0, \quad t\geq 0.
	\end{split}
	\end{cases}
\end{equation*}
where $N$ is a Poisson process with intensity $\lambda$, and $\tilde{N}_t=N_t-\lambda t.$
\\
As suggested in Remark \ref{rem}, the increment \eqref{increment} can be approached by:\\
\\
$\widehat{\Delta_k K}^N:=$
\begin{equation*}
\begin {aligned}
\inf\Bigg\{x\geq 0:\frac{1}{N}\sum_{i=1}^N h\Bigg(&x+\Big(\tilde{X}_{T_{k-1}}^{\tilde{\mu}^N}\Big)^i+\frac{T}{n} \Bigg(b\bigg(\Big(\tilde{X}_{T_{k-1}}^{\tilde{\mu}^N}\Big)^i\bigg)-\lambda F\bigg(\Big(\tilde{X}_{T_{k-1}}^{\tilde{\mu}^N}\Big)^j\bigg)+\frac{\sqrt{T}}{\sqrt{n}}\sigma\bigg(\Big(\tilde{X}_{T_{k-1}}^{\tilde{\mu}^N}\Big)^i\bigg)G^i\\&+F\bigg(\Big(\tilde{X}_{T_{k-1}}^{\tilde{\mu}^N}\Big)^i\bigg)H^i\Bigg)\geq0\Bigg\},
\end {aligned}
\end{equation*}
where $G^j\sim \mathcal{N}(0,1)$ and $H^j\sim \mathcal{P}(\lambda(T/n))$ and are i.i.d.

\begin{algorithm}
\caption{Particle approximation}\label{alg_part}
\begin{algorithmic}[1]
\sFor {$1\leq j\leq N$}{
\State$\bigg(\Big(\tilde{X}_{0}^{\tilde{\mu}^N}\Big)^j,\Big(\tilde{U}_{0}^{\tilde{\mu}^N}\Big)^j,\hat{\mu}^N_0\bigg)=(x,x,\delta_x)$}
\sFor {$1\leq k\leq n$}{
\sFor {$1\leq j\leq N$}{\State $G^j\sim \mathcal{N}(0,1)$
                        \State $H^j\sim \mathcal{P}(\lambda(T/n))$
                        \State $\Big(\tilde{U}_{T_{k}}^{\tilde{\mu}^N}\Big)^j=\Big(\tilde{U}_{T_{k-1}}^{\tilde{\mu}^N}\Big)^j+(T/n)\Bigg(b\bigg(\Big(\tilde{X}_{T_{k-1}}^{\tilde{\mu}^N}\Big)^j\bigg)-\lambda F\bigg(\Big(\tilde{X}_{T_{k-1}}^{\tilde{\mu}^N}\Big)^j\bigg)\Bigg)$
                        \State ~~~~~~~~~~~~~~~$+\sqrt{(T/n)}\sigma\bigg(\Big(\tilde{X}_{T_{k-1}}^{\tilde{\mu}^N}\Big)^j\bigg)G^j+F\bigg(\Big(\tilde{X}_{T_{k-1}}^{\tilde{\mu}^N}\Big)^j\bigg)H^j$}
 \State$\tilde{\mu}^N_{T_k}=N^{-1}\sum_{j=1}^{N}\delta_{(\tilde{U}_{T_k}^{\tilde{\mu}^N})^j}$                       
 \State$\Delta_k \hat{K}^N=\sup_{l\leq k}G_0(\tilde{\mu}^N_{T_l})-\sup_{l\leq k-1}G_0(\tilde{\mu}^N_{T_l})$                       
\sFor {$1\leq j\leq N$}{\State                       $\Big(\tilde{X}_{T_{k}}^{\tilde{\mu}^N}\Big)^j=\Big(\tilde{X}_{T_{k-1}}^{\tilde{\mu}^N}\Big)^j+(T/n)\Bigg(b\bigg(\Big(\tilde{X}_{T_{k-1}}^{\tilde{\mu}^N}\Big)^j\bigg)-\lambda F\bigg(\Big(\tilde{X}_{T_{k-1}}^{\tilde{\mu}^N}\Big)^j\bigg)\Bigg)$
                        \State ~~~~~~~~~~~~~~~$+\sqrt{(T/n)}\sigma\bigg(\Big(\tilde{X}_{T_{k-1}}^{\tilde{\mu}^N}\Big)^j\bigg)G^j+F\bigg(\Big(\tilde{X}_{T_{k-1}}^{\tilde{\mu}^N}\Big)^j\bigg)H^j+\Delta_k\hat{K}^N$}
}

\end{algorithmic}
\end{algorithm}
Indeed, using the same kind of arguments as in the sketch of the proof of Theorem $\ref{thrm_exacte}$, one can show that the increments of the approximated reflection process are equals to the approximation of the increments:
\begin{equation*}
\forall k\in\{1,\cdots n\}:~~ \widehat{\Delta_k K}^N=\Delta_k\hat{K}^N.
\end{equation*}

\newpage 

Returning to the general case \eqref{eq:main}, we can see in \cite{YS12}, $N=\{N(t):=N(E\times[0,t])\}$ is a stochastic process with intensity $\lambda$ that counts the number of jumps until some given time. The Poisson random measure $N(dz,dt)$ generates a sequence of pairs $\{(\iota_i,\xi_i),i\in \{1,2,\cdots,N(T)\}\}$ for a given finite positive constant $T$ if $\lambda<\infty$. Here $\{\iota_i,i\in \{1,2,\cdots,N(T)\}\}$ is a sequence of increasing nonnegative random variables representing the jump times of a standard Poisson process with intensity $\lambda$, and $\{\xi_i,i\in \{1,2,\cdots,N(T)\}\}$ is a sequence of independent identically distributed random variables, where $\xi_i$ is distributed according to $f(z)$, where $\lambda(dz)dt=\lambda f(z)dzdt$.
Then, the numerical approximation can equivalently be the following form
\begin{equation*}
\begin{aligned}
\bar{X}^j_{T_k}&=\bar{X}^j_{T_{k-1}}+\frac{T}{n}\bigg(b(\bar{X}^j_{T_{k-1}})-\int_E\lambda F(\bar{X}^j_{T_{k-1}},z)f(z)dz\bigg)+\sqrt{\frac{T}{n}}\sigma(\bar{X}^j_{T_{k-1}})G^j\\&~~~~+\sum_{i=H^j_{T_{k-1}}+1}^{H^j_{T_{k}}}F(\bar{X}^j_{T_{k-1}},\xi_i)+\Delta_k \hat{K}^N,
\end{aligned}
\end{equation*}
$\Delta_k\hat{K}^N=\widehat{\Delta_k K}^N=$
\begin{equation*}
\begin {aligned}
\inf\Bigg\{x\geq 0:\frac{1}{N}\sum_{i=1}^N h\Bigg(&x+\bar{X}^j_{T_{k-1}}+\frac{T}{n}\bigg(b(\bar{X}^j_{T_{k-1}})-\int_E\lambda F(\bar{X}^j_{T_{k-1}},z)f(z)dz\bigg)+\sqrt{\frac{T}{n}} \sigma(\bar{X}^j_{T_{k-1}})G^j\\&~~~~+\sum_{i=H^j_{T_{k-1}}+1}^{H^j_{T_{k}}}F(\bar{X}^j_{T_{k-1}},\xi_i)\Bigg)\geq0\Bigg\},
\end {aligned}
\end{equation*}
where $G^j\sim \mathcal{N}(0,1)$ and $H^j\sim \mathcal{P}(\lambda(T/n))$ and are i.i.d.

\subsection{Scheme error.}

\begin{proposition}
\label{cv_num}
\begin{enumerate}
\item[(i)] Let $T>0$, $N$ and $n$ be two non-negative integers and let Assumptions $\ref{lip}$, $\ref{bilip}$ and $\ref{int}$ hold. There exists a constant $C$, depending on $T$, $b$, $\sigma$, $F$, $h$ and $X_0$ but independent of $N$, such that: for all $i=1,\ldots,N$
$$\e\bigg[\sup_{s\le t}\big|X^i_s-\tilde{X}^i_s\big|^2\bigg]\leq C\bigg(n^{-1}+N^{-1/2}\bigg).$$
\item[(ii)] Moreover, if Assumption \ref{reg} is in force, there exists a constant $C$, depending on $T$, $b$, $\sigma$, $F$, $h$ and $X_0$ but independent of $N$, such that: for all $i=1,\ldots,N$
$$\e\bigg[\sup_{s\le t}\big|X^i_s-\tilde{X}^i_s\big|^2\bigg]\leq C\bigg(n^{-1}+N^{-1}\bigg).$$
\end{enumerate}
\end{proposition}

\begin{proof}
Let us fix $i\in {1,\ldots,N}$ and $T>0$. We have, for $t\leq T$,
\begin{equation*}
\begin{aligned}
\bigg|X^i_s-\tilde{X}^i_s\bigg|& \leq \bigg|\int_0^s b(X^i_{r^-})- b(\tilde{X}^i_{\underline{r}^-})dr \bigg| + \bigg|\int_0^s \bigg(\sigma(X^i_{r^-})-\sigma(\tilde{X}^i_{\underline{r}^-})\bigg) dB^i_r\bigg|  \\&~~
~~+ \bigg|\int_0^s\int_E\bigg(F(X^i_{r^-},z) -F(\tilde{X}^i_{\underline{r}^-},z)\bigg) \tilde{N}^i(dr,dz)\bigg| + \sup_{r\le s} \big|{G}_0(\mu_r^N)-{G}_0(\tilde{\mu}^N_{\underline{r}})\big|.
\end{aligned}
\end{equation*}
Hence, using Assumption \ref{lip}, Cauchy-Schwartz, Doob and BDG inequalities  we get:
\begin{equation}
\label{X1}
\begin{aligned}
\e\bigg[\sup_{s\le t}\big|X^i_s-\tilde{X}^i_s\big|^2\bigg] & \leq 4\e\bigg[\sup_{s\le t}\Bigg\{\bigg|\int_0^s\bigg(b(X^i_{r^-})- b(\tilde{X}^i_{\underline{r}^-})\bigg) dr\bigg|^2 + \bigg|\int_0^s \bigg(\sigma(X^i_{r^-})-\sigma(\tilde{X}^i_{\underline{r}^-})\bigg) dB^i_r\bigg|^2  \\&~~
~~+ \bigg|\int_0^s\int_E\bigg(F(X^i_{r^-},z) -F(\tilde{X}^i_{\underline{r}^-},z)\bigg) \tilde{N}^i(dr,dz)\bigg|^2 + \sup_{r\le s} \big|{G}_0(\mu_r^N)-{G}_0(\tilde{\mu}^N_{\underline{r}})\big|^2\Bigg\}\bigg] \\&
\leq C\Bigg\{\e \bigg[t\int_0^t\bigg|b(X^i_{s^-})- b(\tilde{X}^i_{\underline{s}^-}) \bigg|^2ds\bigg] + \e\bigg[\int_0^t \bigg|\sigma(X^i_{s^-})-\sigma(\tilde{X}^i_{\underline{s}^-})\bigg|^2 ds\bigg]\\&~~
~~+\e \bigg[\int_0^t\int_E\bigg|F(X^i_{s^-},z) -F(\tilde{X}^i_{\underline{s}^-},z)\bigg|^2 \lambda(dz)ds\bigg] + \e\bigg[\sup_{s\le t} \big|{G}_0(\mu_s^N)-{G}_0(\tilde{\mu}^N_{\underline{s}})\big|^2\bigg]\Bigg\} \\&
\leq C\Bigg\{TC_1\int_0^t \e\bigg[\big|X^i_s-\tilde{X}^i_{\underline{s}}\big|^2\bigg] ds+ C_1\int_0^t \e\bigg[\big|X^i_s-\tilde{X}^i_{\underline{s}}\big|^2\bigg] ds\\&~~
~~+C_1\int_0^t \e\bigg[\big|X^i_s-\tilde{X}^i_{\underline{s}}\big|^2\bigg] ds+ \e\bigg[\sup_{s\le t} \big|{G}_0(\mu_s^N)-{G}_0(\tilde{\mu}^N_{\underline{s}})\big|^2\bigg]\Bigg\}\\&
\leq C\int_0^t \e\bigg[\big|X^i_s-\tilde{X}^i_{\underline{s}}\big|^2\bigg] ds +4\e\bigg[\sup_{s\le t} \big|{G}_0(\mu_s^N)-{G}_0(\tilde{\mu}^N_{\underline{s}})\big|^2\bigg].
\end{aligned}
\end{equation}
Denoting by  $(\mu^i_{t})_{0\le t\le T}$ the family of marginal laws of $(U^i_{t})_{0\le t\le T}$ and $(\tilde{\mu}^i_{\underline{t}})_{0\le t\le T}$ the family of marginal laws of $(\tilde{U}^i_{t})_{0\le t\le T}$, we have
\begin{equation*}
\begin{aligned}
\e\bigg[\sup_{s\le t} \big|{G}_0(\mu_s^N)-{G}_0(\tilde{\mu}^N_{\underline{s}})\big|^2\bigg]&
\leq3\Bigg\{\e\bigg[\sup_{s\le t} \big|{G}_0(\mu_s^N)-{G}_0(\mu_s^i)\big|^2\bigg]
+\sup_{s\le t} \big|{G}_0(\mu_s^i)-{G}_0(\tilde{\mu}^i_{\underline{s}})\big|^2\\&~~
~~+\e\bigg[\sup_{s\le t} \big|{G}_0(\tilde{\mu}^i_{\underline{s}})-{G}_0(\tilde{\mu}^N_{\underline{s}})\big|^2\bigg]\Bigg\},
\end{aligned}
\end{equation*}
and from Lemma $\ref{wiss}$,
\begin{equation*}
\begin{aligned}
~~~~~~~~~~~~~~~~~~~~~~~~~~~~~~~&
\leq3\Bigg\{{1\over{m^2}}\e\bigg[\sup_{s\le t}\left|\int h(\bar{G}_0(\mu^i_s)+\cdot) (d\mu_s^N-d\mu_s^i)\right|^2\bigg]
+\bigg(\frac{M}{m}\bigg)^2\sup_{s\le t} W_1^2(\mu_s^i,\tilde{\mu}^i_{\underline{s}})\\&~~
~~+{1\over{m^2}}\e\bigg[\sup_{s\le t}\left|\int h(\bar{G}_0(\tilde{\mu}^i_{\underline{s}})+\cdot) (d\tilde{\mu}^N_{\underline{s}}-d\tilde{\mu}^i_{\underline{s}})\right|^2\bigg]\Bigg\}\\&
\leq C\Bigg\{\e\bigg[\sup_{s\le t}\left|\int h(\bar{G}_0(\mu^i_s)+\cdot) (d\mu_s^N-d\mu_s^i)\right|^2\bigg]
+\sup_{s\le t} \e\bigg[\bigg|U_s^i-\tilde{U}^i_{\underline{s}}\bigg|^2\bigg]\\&~~
~~+\e\bigg[\sup_{s\le t}\left|\int h(\bar{G}_0(\tilde{\mu}^i_{\underline{s}})+\cdot) (d\tilde{\mu}^N_{\underline{s}}-d\tilde{\mu}^i_{\underline{s}})\right|^2\bigg]\Bigg\}.
\end{aligned}
\end{equation*}
\begin{proof}[Proof of (i)]\renewcommand{\qedsymbol}{}
Following the Proof of (i) in Theorem \ref{cv_part}, we obtain
$$\e\bigg[\sup_{s\le t}\left|\int h(\bar{G}_0(\mu^i_s)+\cdot) (d\mu_s^N-d\mu_s^i)\right|^2\bigg]\leq C\e\bigg[\sup_{s\le t} W_1^2(\mu_s^N,\mu_s^i)\bigg]\leq CN^{-1/2},$$
$$\e\bigg[\sup_{s\le t}\left|\int h(\bar{G}_0(\tilde{\mu}^i_{\underline{s}})+\cdot) (d\tilde{\mu}^N_{\underline{s}}-d\tilde{\mu}^i_{\underline{s}})\right|^2\bigg]\leq C\e\bigg[\sup_{s\le t} W_1^2(\tilde{\mu}^i_{\underline{s}},\tilde{\mu}^N_{\underline{s}})\bigg]\leq CN^{-1/2}.$$
From which, we can derive the inequality
\begin{equation*}
\begin{aligned}
\e\bigg[\sup_{s\le t} \big|{G}_0(\mu_s^N)-{G}_0(\tilde{\mu}^N_{\underline{s}})\big|^2\bigg]&
\leq C_1\sup_{s\le t} \e\bigg[\bigg|U_s^i-\tilde{U}^i_{\underline{s}}\bigg|^2\bigg]+C_2N^{-1/2}\\&
\leq C_1\bigg\{\sup_{s\le t} \e\bigg[\bigg|U_s^i-\tilde{U}^i_{s}\bigg|^2\bigg]+\sup_{s\le t} \e\bigg[\bigg|\tilde{U}^i_{s}-\tilde{U}^i_{\underline{s}}\bigg|^2\bigg]\bigg\}
+C_2N^{-1/2}.
\end{aligned}
\end{equation*}
For the first term of the right hand side, we can observe that,
\begin{equation*}
\begin{aligned}
\sup_{s\le t}\e\bigg[\bigg|U_s^i-\tilde{U}^i_{s}\bigg|^2\bigg]\bigg]&
\leq \e\bigg[\sup_{s\le t}\bigg|U_s^i-\tilde{U}^i_{s}\bigg|^2\bigg]\bigg]\\&
\leq 3\e\bigg[\sup_{s\le t}\Bigg\{\bigg|\int_0^s\bigg(b(X^i_{r^-})- b(\tilde{X}^i_{\underline{r}^-})\bigg) dr\bigg|^2 + \bigg|\int_0^s \bigg(\sigma(X^i_{r^-})-\sigma(\tilde{X}^i_{\underline{r}^-})\bigg) dB^i_r\bigg|^2  \\&~~+ \bigg|\int_0^s\int_E\bigg(F(X^i_{r^-},z) -F(\tilde{X}^i_{\underline{r}^-},z)\bigg) \tilde{N}^i(dr,dz)\bigg|^2\bigg\}\bigg] \\&
\leq C\Bigg\{\e \bigg[t\int_0^t\bigg|b(X^i_{s^-})- b(\tilde{X}^i_{\underline{s}^-}) \bigg|^2ds\bigg] + \e\bigg[\int_0^t \bigg|\sigma(X^i_{s^-})-\sigma(\tilde{X}^i_{\underline{s}^-})\bigg|^2 ds\bigg]\\&~~
~~+\e \bigg[\int_0^t\int_E\bigg|F(X^i_{s^-},z) -F(\tilde{X}^i_{\underline{s}^-},z)\bigg|^2 \lambda(dz)ds\bigg]\Bigg\} \\&
\leq C\Bigg\{TC_1\int_0^t \e\bigg[\big|X^i_s-\tilde{X}^i_{\underline{s}}\big|^2\bigg] dr+ 2C_1\int_0^t \e\bigg[\big|X^i_s-\tilde{X}^i_{\underline{s}}\big|^2\bigg] dr\Bigg\}\\&
\leq C\int_0^t \e\bigg[\big|X^i_s-\tilde{X}^i_{\underline{s}}\big|^2\bigg] ds.
\end{aligned}
\end{equation*}
Using Assumption \ref{lip}, the second term $\sup_{s\le t} \e\bigg[\bigg|\tilde{U}^i_{s}-\tilde{U}^i_{\underline{s}}\bigg|^2\bigg]$ becomes
\begin{equation*}
\begin{aligned}
\sup_{s\le t} \e\bigg[\bigg|\tilde{U}^i_{s}-\tilde{U}^i_{\underline{s}}\bigg|^2\bigg]&
\leq 3\sup_{s\le t}\Bigg\{\e\bigg[\bigg|\int_{\underline{s}}^s b(\tilde{X}^i_{\underline{r}^-}) dr\bigg|^2 
+ \bigg|\int_{\underline{s}}^s \sigma(\tilde{X}^i_{\underline{r}^-}) dB^i_r\bigg|^2 
+ \bigg|\int_{\underline{s}}^s \int_E F(\tilde{X}^i_{\underline{r}^-},z) \tilde{N}^i(dr,dz)\bigg|^2\bigg]\Bigg\}\\&
\leq 3\sup_{s\le t}\Bigg\{\e\bigg[\bigg| b(\tilde{X}^i_{\underline{s}}) \bigg|^2 \big|s-\underline{s}\big|^2
+ \bigg|\sigma(\tilde{X}^i_{\underline{s}}) \bigg|^2 \big|B^i_s-B^i_{\underline{s}}\big|^2 
+ \int_{\underline{s}}^s \int_E \bigg|F(\tilde{X}^i_{\underline{r}^-},z)\bigg|^2 \lambda(dz)dr\bigg]\Bigg\}\\&
\leq 3\sup_{s\le t}\Bigg\{\e\bigg[\bigg| b(\tilde{X}^i_{\underline{s}}) \bigg|^2 \big|s-\underline{s}\big|^2
+ \bigg|\sigma(\tilde{X}^i_{\underline{s}}) \bigg|^2 \big|B^i_s-B^i_{\underline{s}}\big|^2 
+  C\int_{\underline{s}}^s(1+|\tilde{X}^i_{\underline{r}^-}|^2) dr\bigg]\Bigg\}\\&
\leq 3\sup_{s\le t}\Bigg\{\bigg(\frac{T}{n}\bigg)^2\e\bigg[\big|\sup_{\underline{s}\leq r\leq s} b(\tilde{X}^i_{\underline{r}^-}) \big|^2\bigg]
+  \e\bigg[\big|B^i_s-B^i_{\underline{s}}\big|^2\bigg] \e\bigg[\big|\sigma(\tilde{X}^i_{\underline{s}}) \big|^2\bigg]
\\&~~~~+C\bigg(\frac{T}{n}\bigg)\e\bigg[\sup_{\underline{s}\leq r\leq s} (1+|\tilde{X}^i_r|^2)\bigg]\Bigg\}\\&
\leq C_1\bigg(\frac{T}{n}\bigg)^2\e\bigg[\sup_{s\le T} \big|b(\tilde{X}^i_s) \big|^2\bigg]
+ C_2 \sup_{s\le t}\e\bigg[\big|B^i_s-B^i_{\underline{s}}\big|^2\bigg] \e\bigg[\sup_{s\le T}\big|\sigma(\tilde{X}^i_s) \big|^2\bigg]\\&~~~~
+ C_3 \bigg(\frac{T}{n}\bigg)\e\bigg[\sup_{s\le T} (1+|\tilde{X}^i_s|^2)\bigg]\\&
\leq C_1\bigg(\frac{T}{n}\bigg)^2\bigg(1+\e\bigg[\sup_{s\le T} \big|\tilde{X}^i_s \big|^2\bigg]\bigg)
+ C_2 \sup_{s\le t}\e\bigg[\big|B^i_s-B^i_{\underline{s}}\big|^2\bigg] \bigg(1+\e\bigg[\sup_{s\le T} \big|\tilde{X}^i_s \big|^2\bigg]\bigg)\\&~~~~
+ C_3  \bigg(\frac{T}{n}\bigg)\bigg(1+\e\bigg[\sup_{s\le T} \big|\tilde{X}^i_s \big|^2\bigg]\bigg),
\end{aligned}
\end{equation*}
and from Proposition $\ref{propriete_1}$, we get
\begin{equation*}
\begin{aligned}
\sup_{s\le t} \e\bigg[\bigg|\tilde{U}^i_{s}-\tilde{U}^i_{\underline{s}}\bigg|^2\bigg]&
\leq C_1\bigg(\frac{T}{n}\bigg)+ C_2 \sup_{s\le t}\e\bigg[\big|B^i_s-B^i_{\underline{s}}\big|^2\bigg].
\end{aligned}
\end{equation*}
Then, by using BDG inequality, we obtain
\begin{equation*}
\begin{aligned}
\sup_{s\le t}\e\bigg[\big|B^i_s-B^i_{\underline{s}}\big|^2\bigg]
=\sup_{s\le t}\e\bigg[\bigg(\int_{\underline{s}}^s dB^i_u\bigg)^2\bigg]
\leq\sup_{s\le t}|s-\underline{s}|
\leq \frac{T}{n}.
\end{aligned}
\end{equation*}
Therefore, we conclude 
\begin{equation}
\label{U}
\begin{aligned}
\sup_{s\le t} \e\bigg[\bigg|\tilde{U}^i_{s}-\tilde{U}^i_{\underline{s}}\bigg|^2\bigg]&
\leq C_1n^{-1}+ C_2 n^{-1}\\&
\leq Cn^{-1},
\end{aligned}
\end{equation}
from which we derive the inequality
\begin{equation}
\label{G}
\begin{aligned}
\e\bigg[\sup_{s\le t} \big|{G}_0(\mu_s^N)-{G}_0(\tilde{\mu}^N_{\underline{s}})\big|^2\bigg]&
\leq C\Bigg\{n^{-1}+N^{-1/2}+\int_0^t \e\bigg[\big|X^i_s-\tilde{X}^i_{\underline{s}}\big|^2\bigg] ds\Bigg\},
\end{aligned}
\end{equation}
and taking into account $\eqref{X1}$ we get
\begin{equation}
\label{X2}
\begin{aligned}
\e\bigg[\sup_{s\le t}\big|X^i_s-\tilde{X}^i_s\big|^2\bigg] & \leq C\Bigg\{n^{-1}+N^{-1/2}+\int_0^t \e\bigg[\big|X^i_s-\tilde{X}^i_{\underline{s}}\big|^2\bigg] ds\Bigg\}.
\end{aligned}
\end{equation}
Since 
\begin{equation*}
\begin{aligned}
\e\bigg[\big|X^i_s-\tilde{X}^i_{\underline{s}}\big|^2\bigg]&
\leq 2\e\bigg[\big|X^i_s-\tilde{X}^i_{s}\big|^2\bigg]+2\e\bigg[\big|\tilde{X}^i_s-\tilde{X}^i_{\underline{s}}\big|^2\bigg]\\&
=2\e\bigg[\big|X^i_s-\tilde{X}^i_{s}\big|^2\bigg]+2\e\bigg[\big|\tilde{U}^i_s-\tilde{U}^i_{\underline{s}}\big|^2\bigg],
\end{aligned}
\end{equation*}
it follows from $\eqref{U}$ and $\eqref{X2}$ that
\begin{equation*}
\begin{aligned}
\e\bigg[\sup_{s\le t}\big|X^i_s-\tilde{X}^i_s\big|^2\bigg] & \leq C\Bigg\{n^{-1}+N^{-1/2}+\int_0^t \e\bigg[\big|X^i_s-\tilde{X}^i_s\big|^2\bigg] ds\Bigg\}.
\end{aligned}
\end{equation*}
and finally, we conclude the proof of (i) with Gronwall's Lemma.
\end{proof}
\begin{proof}[Proof of (ii)]\renewcommand{\qedsymbol}{}
Following the proof of (ii) in Theorem \ref{cv_part}, we obtain
$$\e\bigg[\sup_{s\le t}\left|\int h(\bar{G}_0(\mu^i_s)+\cdot) (d\mu_s^N-d\mu_s^i)\right|^2\bigg]\leq CN^{-1},$$
$$\e\bigg[\sup_{s\le t}\left|\int h(\bar{G}_0(\tilde{\mu}^i_{\underline{s}})+\cdot) (d\tilde{\mu}^N_{\underline{s}}-d\tilde{\mu}^i_{\underline{s}})\right|^2\bigg]\leq CN^{-1}.$$ 
According to the same strategy applied to proof of (i) in Theorem \ref{cv_num}, the result follows easily:
$$\e\bigg[\sup_{s\le t}\big|X^i_s-\tilde{X}^i_s\big|^2\bigg]\leq C\bigg(n^{-1}+N^{-1}\bigg).$$
\end{proof}
\end{proof}

\begin{theorem}
\label{cv}
Let $T > 0$, $N$ and $n$ be two non-negative integers. Let assumptions $\ref{lip}$, $\ref{bilip}$ and $\ref{int}$ hold.
\begin{enumerate}
\item[(i)] There exists a constant $C$, depending on $T$, $b$, $\sigma$, $F$, $h$ and $X_0$ but independent of $N$, such that: for all $i=1,\ldots,N$,
$$\e\bigg[\sup_{t\le T}\big|\bar{X}^i_t-\tilde{X}^i_t\big|^2\bigg]\leq C\bigg(n^{-1}+N^{-1/2}\bigg).$$
\item[(ii)] If in addition $\ref{reg}$ holds, there exists a positive constant $C$, depending on $T$, $b$, $\sigma$, $F$, $h$ and $X_0$ but independent of $N$, such that: for all $i=1,\ldots,N$,
$$\e\bigg[\sup_{t\le T}\big|\bar{X}^i_t-\tilde{X}^i_t\big|^2\bigg]\leq C\bigg(n^{-1}+N^{-1}\bigg).$$
\end{enumerate}
\end{theorem}
\begin{proof}
The proof is straightforward writing
$$\big|\bar{X}^i_t-\tilde{X}^i_t\big|\leq \big|\bar{X}^i_t-{X}^i_t\big|+\big|{X}^i_t-\tilde{X}^i_t\big|,$$
and using Theorem \ref{cv_part} and Proposition \ref{cv_num}.
\end{proof}

\section{Numerical illustrations.}\label{sec:NI}
Throughout this section, we consider, on $[0,T]$ the following sort of processes:
\begin{equation}
	\label{eq-exacte}
	\begin{cases}
	\begin{split}
	& X_t  =X_0-\int_0^t (\beta_s+a_s X_{s^-}) ds + \int_0^t (\sigma_s+\gamma_s X_{s^-}) dB_s  + \int_0^t \int_E c(z)(\eta_s+\theta_s X_{s^-}) \tilde{N}(ds,dz) + K_t,\quad t\geq 0, \\
	& \e[h(X_t)] \geq 0, \quad \int_0^t \e[h(X_s)] \, dK_s = 0, \quad t\geq 0.
	\end{split}
	\end{cases}
\end{equation} 

where $(\beta_t)_{t\geq0}$, $(a_t)_{t\geq0}$, $(\sigma_t)_{t\geq0}$, $(\gamma_t)_{t\geq0}$, $(\eta_t)_{t\geq0}$ and $(\theta_t)_{t\geq0}$ are bounded adapted processes. This sort of processes allow us to make some explicit computations leading us to illustrate the algorithm. Our results are presented for different diffusions and functions $h$ that are summarized below.

\textbf{Linear constraint.}
We first consider cases where $h:\mathbb{R}\ni x \longmapsto x-p\in\mathbb{R}.$

\begin{itemize}
\item [Case (i)] Drifted Brownian motion and compensated Poisson process: $\beta_t=\beta>0$, $a_t=\gamma_t=\theta_t=0$, $\sigma_t=\sigma>0$, $\eta_t=\eta>0$, $X_0=x_0\geq p$, $c(z)=z$ and $$f(z)=\frac{1}{\sqrt{2\pi z}}\exp\bigg(-\frac{(\ln z)^2}{2}\bigg)\bf{1}_{\{0<z\}}.$$
We have $$K_t=(p+\beta t-x_0)^+,$$
and $$X_t=X_0-(\beta+\lambda\sqrt{e})t+\sigma B_t+\sum_{i=0}^{N_t}\eta \xi_i+K_t,$$ 
where $N_t\sim \mathcal{P}(\lambda t)$ and $\xi_i\sim lognormal(0,1).$

\item [Case (ii)] Black and Scholes process: $\beta_t=\sigma_t=\eta_t=0$, $a_t=a>0$, $\gamma_t=\gamma>0$, $\theta_t=\theta>0$, $c(z)=\delta_1(z)$. Then\\
\\ $K_t=ap(t-t^*)1_{t\geq t^*}$, where $t^*=\frac{1}{a}(\ln(x_0)-\ln(p))$,\\
and
$$X_t=Y_t+Y_t\int_0^tY_s^{-1}dK_s,$$
where $Y$ is the process defined by:
$$Y_t=X_0\exp\Big(-(a+\gamma^2/2+\lambda\theta)t+\gamma B_t\Big)(1+\theta)^{N_t}.$$
\end{itemize}

\textbf{Nonlinear constraint.}
Secondly, we illustrate the case of non-linear function $h$:
$$h:\mathbb{R}\ni x\longmapsto x+\alpha \sin(x)-p\in\mathbb{R},~~-1<\alpha<1,$$
and we illustrate this case with 
\begin{itemize}
\item [Case (iii)] Ornstein Uhlenbeck process: $\beta_t=\beta>0$, $a_t=a>0$, $\gamma_t=\theta_t=0$, $\sigma_t=\sigma>0$, $\eta_t=\eta>0$, $X_0=x_0$ with $x_0>|\alpha|+p$, $c(z)=\delta_1(z)$. We obtain
$$\mathrm{d}K_t=e^{-at}\mathrm{d}\sup_{s\leq t}(F^{-1}_s(0))^+,$$
where for all $t$ in $[0,T],$
\begin{equation*}
\begin{aligned}
F_t:\mathbb{R}\ni x\longmapsto &\Bigg\{e^{-at}\bigg(x_0-\beta\bigg(\frac{e^{at}-1}{a}\bigg)+x\bigg)
+\alpha\exp\bigg(-e^{-at}\frac{\sigma^2}{2a}\sinh(at)\bigg)
\\&\times\Bigg[\frac{1}{2}\bigg(\exp(\lambda t(e^{i\eta}-1))+\exp(\lambda t(e^{-i\eta}-1))\bigg)\sin\bigg(e^{-at}\bigg(x_0-(\beta+\lambda\eta)\bigg(\frac{e^{-at}-1}{a}\bigg)+x\bigg)\bigg)
\\&+\frac{1}{2i}\bigg(\exp(\lambda t(e^{i\eta}-1))-\exp(\lambda t(e^{-i\eta}-1))\bigg)\cos\bigg(e^{-at}\bigg(x_0-(\beta+\lambda\eta)\bigg(\frac{e^{-at}-1}{a}\bigg)+x\bigg)\bigg)\Bigg]
\\&-p\Bigg\}
\end{aligned}
\end{equation*}
\end{itemize}

\begin{remark}
These examples have been chosen in such a way that we are able to give an analytic form of the reflecting process $K$. This enables us to compare numerically the “true” process $K$ and its empirical approximation $\hat{K}$ . When an exact simulation of the underlying process is available, we
compute the approximation rate of our algorithm.
\end{remark}

\subsection{ Proofs of the numerical illustrations.}
In order to have closed, or almost closed, expression for the compensator $K$ we introduce the process $Y$ solution to the non-reflected SDE
\begin{equation*}
Y_t  =X_0-\int_0^t (\beta_s+a_s Y_{s^-}) ds + \int_0^t (\sigma_s+\gamma_s Y_{s^-}) dB_s  + \int_0^t\int_E c(z)(\eta_s+\theta_s Y_{s^-}) \tilde{N}(ds,dz).
\end{equation*}
By letting $A_s=\int_0^ta_sds$ and applying Itô’s formula on $e^{A_t}X_t$ and $e^{A_t}Y_t$, we get
\begin{equation*}
\begin{aligned}
e^{A_t}X_t&=X_0+\int_0^t e^{A_s}X_s a_s ds+\int_0^t e^{A_s}(-\beta_s-a_s X_{s^-}) ds + \int_0^t e^{A_s}(\sigma_s+\gamma_s X_{s^-}) dB_s  \\&~~~~+ \int_0^t\int_E e^{A_s} c(z)(\eta_s+\theta_s X_{s^-}) \tilde{N}(ds,dz)+\int_0^te^{A_s}dK_s\\&
=X_0-\int_0^t e^{A_s}\beta_s ds + \int_0^t e^{A_s}(\sigma_s+\gamma_s X_{s^-}) dB_s  + \int_0^t\int_E e^{A_s} c(z)(\eta_s+\theta_s X_{s^-}) \tilde{N}(ds,dz)+\int_0^te^{A_s}dK_s,
\end{aligned}
\end{equation*}
in the same way,
\begin{equation*}
\begin{aligned}
e^{A_t}Y_t=X_0-\int_0^t e^{A_s}\beta_s ds + \int_0^t e^{A_s}(\sigma_s+\gamma_s Y_{s^-}) dB_s  + \int_0^t\int_E e^{A_s} c(z)(\eta_s+\theta_s Y_{s^-}) \tilde{N}(ds,dz),
\end{aligned}
\end{equation*}
and so
\begin{equation*}
\begin{aligned}
X_t=Y_t+e^{-A_t}\int_0^te^{A_s}dK_s+ e^{-A_t}\int_0^t e^{A_s}\gamma_s (X_{s^-}+Y_{s^-}) dB_s  +e^{-A_t} \int_0^t\int_E e^{A_s} c(z)\theta_s(X_{s^-}+Y_{s^-}) \tilde{N}(ds,dz).
\end{aligned}
\end{equation*}

\begin{remark}
\label{E(Y)}
In all of cases, we have $a_t=a$ i.e. $A_t=at$, so we get
\begin{equation*}
\begin{aligned}
\e[Y_t]&=\e\bigg[e^{-at}\bigg(x_0-\int_0^t e^{as}\beta ds+ \int_0^t e^{as}(\sigma_s+\gamma_s Y_{s^-}) dB_s  + \int_0^t\int_E e^{as} c(z)(\eta_s+\theta_s Y_{s^-}) \tilde{N}(ds,dz)\bigg)\bigg]\\&
=e^{-at}\bigg(x_0-\int_0^t e^{as}\beta ds\bigg)\\&
=e^{-at}\bigg(x_0-\beta \bigg(\frac{e^{at}-1}{a}\bigg)\bigg).
\end{aligned}
\end{equation*}
\end{remark}

\begin{proof}[Proof of assertions (i)]
From Proposition $\ref{propdensite}$ and Remark $\ref{E(Y)}$, we have
\begin{equation*}
\begin{aligned}
k_t&=\beta \bf{1}_{\e(X_t)=p}\\&
=\beta \bf{1}_{\e(Y_t)+K_t-p=0}\\&
=\beta \bf{1}_{x_0-\beta t+K_t-p=0},
\end{aligned}
\end{equation*}
so, we obtain that
\begin{equation*}
\begin{aligned}
K_t&=\int_0^t k_t dt\\&
=\int_0^t \beta \bf{1}_{K_t=p+\beta t-x_0} dt,
\end{aligned}
\end{equation*}
and as $K_t\geq 0$, we conclude that
$$K_t=(p+\beta t-x_0)^+.$$
Next, we have 
$$f(z)=\frac{1}{\sqrt{2\pi z}}\exp\bigg(-\frac{(\ln z)^2}{2}\bigg),$$
the density function of a lognormal random variable,
so we can obtain 
$$\int_E \eta z \lambda(dz)=\lambda\eta\int_E z f(z) dz=\lambda\eta \e(\xi)$$
where $\xi\sim lognormal(0,1)$,
and we conclude that 
$$\int_E \eta z \lambda(dz)=\lambda\eta\sqrt{e}.$$
Finally, we deduce the exact solution
$$X_t=X_0-(\beta+\lambda\sqrt{e}) t+\sigma B_t+\sum_{i=0}^{N_t}\eta \xi_i+K_t,$$
where $N_t\sim \mathcal{P}(\lambda t)$ and $\xi_i\sim lognormal(0,1).$
\end{proof}

\begin{proof}[Proof of assertions (ii)]
In this case, and using the same Proposition and Remark, we have
\begin{equation*}
\begin{aligned}
k_t&=(\e(-aX_t))^- \bf{1}_{\e(X_t)=p},
\end{aligned}
\end{equation*}
which
\begin{equation*}
\begin{aligned}
\e(X_t)=p&\Longleftrightarrow \e(Y_t)-p+e^{-at}\int_0^t e^{as}dK_s=0\\&
\Longleftrightarrow -x_0e^{-at}+p=e^{-at}\int_0^t e^{as}dK_s\\&
\Longleftrightarrow K_s=ap,
\end{aligned}
\end{equation*}
and 
\begin{equation*}
\begin{aligned}
K_t\geq0&\Longleftrightarrow -x_0e^{-at}+p\geq0\\&
\Longleftrightarrow e^{-at}\leq\frac{p}{x_0}\\&
\Longleftrightarrow t\geq\frac{1}{a}(\ln(x_0)-\ln(p)):=t^*.
\end{aligned}
\end{equation*}
So, we conclude that
$K_t=ap(t-t^*)1_{t\geq t^*}$, where $t^*=\frac{1}{a}(\ln(x_0)-\ln(p))$.\\
Next, by the definition of the process $Y_t$ in this case,
\begin{equation*}
dY_t  =-a Y_{t^-} dt + \gamma Y_{t^-} dB_s  + \theta Y_{t^-} d\tilde{N}_t,
\end{equation*}
we have 
$$Y_t=X_0\exp\Big(-(a+\gamma^2/2+\lambda\theta)t+\gamma B_t\Big)(1+\theta)^{N_t}.$$
Thanks to It\^o formula we get
\begin{equation*}
\begin{aligned}
d\bigg(\frac{1}{Y_t}\bigg)&=-\frac{1}{Y_t^2}dY_t+\frac{1}{2}\bigg(\frac{2}{Y_t^3}\bigg)\gamma^2Y_t^2dt+d\sum_{s\leq t}\bigg(\frac{1}{Y_{s^-}+\Delta Y_s}-\frac{1}{Y_{s^-}}+\frac{1}{Y_{s^-}^2}\Delta Y_s\bigg)\\&
=\frac{a}{Y_t}dt-\frac{\gamma}{Y_t}dB_t-\frac{\theta}{Y_{t^-}}d\tilde{N}_t+\frac{\gamma^2}{Y_t}dt+d\sum_{s\leq t}\bigg(\frac{1}{(1+\theta)Y_{s^-}}-\frac{1}{Y_{s^-}}+\frac{\theta}{Y_{s^-}}\bigg),
\end{aligned}
\end{equation*}
and so
\begin{equation*}
\begin{aligned}
dY_t^{-1}&
=(a+\gamma^2)Y_t^{-1}dt-\gamma Y_t^{-1}dB_t-\theta Y_{t^-}^{-1}d\tilde{N}_t+\bigg(\frac{\theta^2}{1+\theta}\bigg)d\sum_{s\leq t}Y_{s^-}^{-1}\\&
=\bigg(a+\gamma^2+\frac{\lambda\theta^2}{1+\theta}\bigg)Y_t^{-1}dt-\gamma Y_t^{-1}dB_t-\bigg(\frac{\theta}{1+\theta}\bigg) Y_{t^-}^{-1}d\tilde{N}_t.
\end{aligned}
\end{equation*}
Then, using integration by parts formula, we obtain
\begin{equation*}
\begin{aligned}
d(X_t Y_t^{-1})&
=X_{t^-} dY_t^{-1}+Y_{t^-}^{-1}dX_t+d[X,Y^{-1}]_t \\&
=(a+\gamma^2)X_tY_t^{-1}dt-\gamma X_tY_t^{-1}dB_t-\theta X_{t^-}Y_{t^-}^{-1}d\tilde{N}_t+\bigg(\frac{\theta^2}{1+\theta}\bigg)d\sum_{s\leq t}X_{s^-}Y_{s^-}^{-1}\\&~~~~-aX_t Y_t^{-1}dt+\gamma X_t Y_t^{-1}dB_t+\theta X_{t^-} Y_{t^-}^{-1}d\tilde{N}_t+Y_t^{-1}dK_t\\&~~~~-\gamma^2X_tY_t^{-1}dt-\bigg(\frac{\theta^2}{1+\theta}\bigg)d\sum_{s\leq t}X_{s^-}Y_{s^-}^{-1}\\&
=Y_t^{-1}dK_t.
\end{aligned}
\end{equation*}
Finally, we deduce that 
$$X_t=Y_t+Y_t\int_0^tY_s^{-1}dK_s.$$
\end{proof}

\begin{proof}[Proof of assertions (iii)]
In that case, we have
\begin{equation*}
\begin{aligned}
Y_t&=e^{-at}\bigg(x_0-\beta \bigg(\frac{e^{at}-1}{a}\bigg)\bigg)+ \sigma_s e^{-at}\int_0^t e^{as}dB_s  + e^{-at}\int_0^t \eta_s e^{as}  d\tilde{N}_s
\\&=e^{-at}\bigg(x_0-(\beta+\lambda\eta) \bigg(\frac{e^{at}-1}{a}\bigg)\bigg)+ \sigma_s e^{-at}\int_0^t e^{as}dB_s  + e^{-at}\int_0^t \eta_s e^{as}  dN_s
\\&:=f_t+G_t+F_t,
\end{aligned}
\end{equation*}
and
$$X_t=Y_t+e^{-at}\bar{K}_t,~~~~\bar{K}_t=\int_0^te^{as}dK_s.$$
Hence
\begin{equation*}
\begin{aligned}
h(X_t)&=Y_t+e^{-at}\bar{K}_t+\alpha\sin(Y_t+e^{-at}\bar{K}_t)-p\\&
=Y_t+e^{-at}\bar{K}_t+\alpha\Big(\sin(Y_t)\cos(e^{-at}\bar{K}_t)+\cos(Y_t)\sin(e^{-at}\bar{K}_t)\Big)-p\\&
=Y_t+e^{-at}\bar{K}_t+\alpha\Big[\cos(e^{-at}\bar{K}_t)\Big\{\sin(f_t)\cos(G_t)\cos(F_t)+\cos(f_t)\sin(G_t)\cos(F_t)\\&~~~~+\cos(f_t)\cos(G_t)\sin(F_t)-\sin(f_t)\sin(G_t)\sin(F_t)\Big\}+\sin(e^{-at}\bar{K}_t)\Big\{\cos(f_t)\cos(G_t)\cos(F_t)\\&~~~~-\sin(f_t)\sin(G_t)\sin(F_t)-\sin(f_t)\cos(G_t)\sin(F_t)-\cos(f_t)\sin(G_t)\sin(F_t)\Big\}\Big]-p.\\&
\end{aligned}
\end{equation*}
On one side, since $G_t$ is a centered gaussian random variable with variance $V$ given by
\begin{equation*}
	V=\sigma^2\frac{1-e^{-2at}}{2a}=\sigma^2e^{-at}\frac{\sinh(at)}{a},
\end{equation*} 
we obtain that $$\e[e^{iG_t}]=e^{-V/2},$$
$$\e[\sin(G_t)]=\e\bigg[\frac{e^{iG_t}-e^{-iG_t}}{2i}\bigg]=0,$$
and
$$\e[\cos(G_t)]=\e\bigg[\frac{e^{iG_t}+e^{-iG_t}}{2}\bigg]=\e(e^{iG_t})=\exp\bigg(-e^{-at}\frac{\sigma^2}{2a}\sinh(at)\bigg)=:g(t).$$

On the other side, 
\begin{equation*}
\begin{aligned}
\e[e^{iF_t}]=\e\bigg[\exp\bigg(i\eta e^{-at}\int_0^t e^{as}dN_s\bigg)\bigg],
\end{aligned}
\end{equation*}
by taking \textquoteleft $a$\textquoteright small, we get
\begin{equation*}
\begin{aligned}
\e[e^{iF_t}]&\approx\e\bigg[\exp\bigg(i\eta \int_0^t dN_s\bigg)\bigg]\\&
\approx\e\Big[\exp\Big(i\eta N_t\Big)\Big]\\&
\approx\exp\Big(\lambda  t(e^{i\eta}-1)\Big),
\end{aligned}
\end{equation*}
and so
\begin{equation*}
\e[\sin(F_t)]\approx\frac{\exp\Big(\lambda  t(e^{i\eta}-1)\Big)-\exp\Big(\lambda  t(e^{-i\eta}-1)\Big)}{2i}=:m(t),
\end{equation*}

\begin{equation*}
\e[\cos(F_t)]\approx\frac{\exp\Big(\lambda  t(e^{i\eta}-1)\Big)+\exp\Big(\lambda  t(e^{-i\eta}-1)\Big)}{2}=:n(t).
\end{equation*}

Using Remark \ref{E(Y)}, we conclude that, for small \textquoteleft $a$\textquoteright,
\begin{equation*}
\begin{aligned}
\e[h(X_t)]&\approx\e[Y_t]+e^{-at}\bar{K}_t+\alpha\Big(g(t)m(t)\cos(f_t+e^{-at}\bar{K}_t)+g(t)n(t)\sin(f_t+e^{-at}\bar{K}_t)\Big)-p\\&:=F_t(\bar{K}_t).
\end{aligned}
\end{equation*}
 
Therefore,
$$\bar{K}_t=\sup_{s\leq t}\Big(F_s^{-1}(0)\Big)^+~~~~and~~~~ dK_t=e^{-at}d\sup_{s\leq t}\Big(F_s^{-1}(0)\Big)^+.$$
\end{proof}

\subsection{Illustrations.}

This computation works as follows. Let $0 = T_0 < T_1 <\cdots< T_n = T $ be
a subdivision of $[0, T ]$ of step size $T/n$, $n$ being a positive integer, let $X$ be the unique solution of the MRSDE \eqref{eq-exacte} and let, for a given $i$, $(\tilde{X}^i_{T_k})_{0\leq k\leq n}$ be its numerical approximation given by Algorithm 1. For a given integer $L$, we draw $(\bar{X}^l)_{0\leq l\leq L}$ and $(\tilde{X}^{i,l})_{0\leq l\leq L}$, $L$ independent copies of $X$ and $\tilde{X}^i$. We then approximate the $\mathbb{L}^2$-error of Theorem~\ref{cv} by:
\begin{equation*}
\hat{E}=\frac{1}{L}\sum_{l=1}^{L}\max_{0\leq k\leq n}\left|\bar{X}^l_{T_k}-\tilde{X}^{i,l}_{T_k}\right|^2.
\end{equation*}

\begin{figure}[h!]
  \centering
	\includegraphics[scale=0.5]{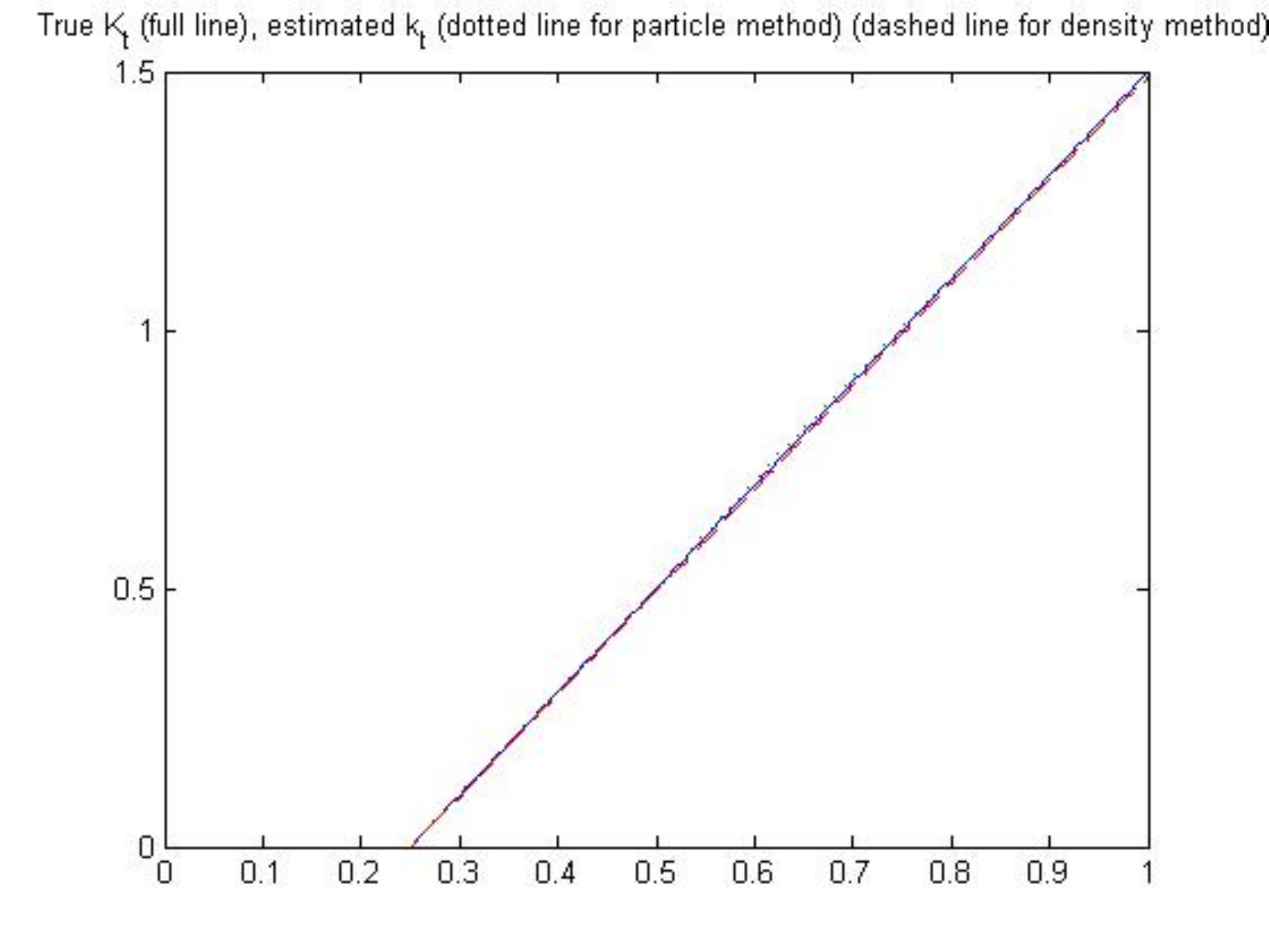}
	\caption{Case (i). $n = 500$, $N = 100000$, $T = 1$, $\beta = 2$, $\sigma = 1$, $\lambda=5$, $x_0 =1$, $p = 1/2$.}
	\label{fig:un}
	
\end{figure}

%

\begin{figure}[h!]
	\label{fig:deux}
  \centering
  \includegraphics[scale=0.5]{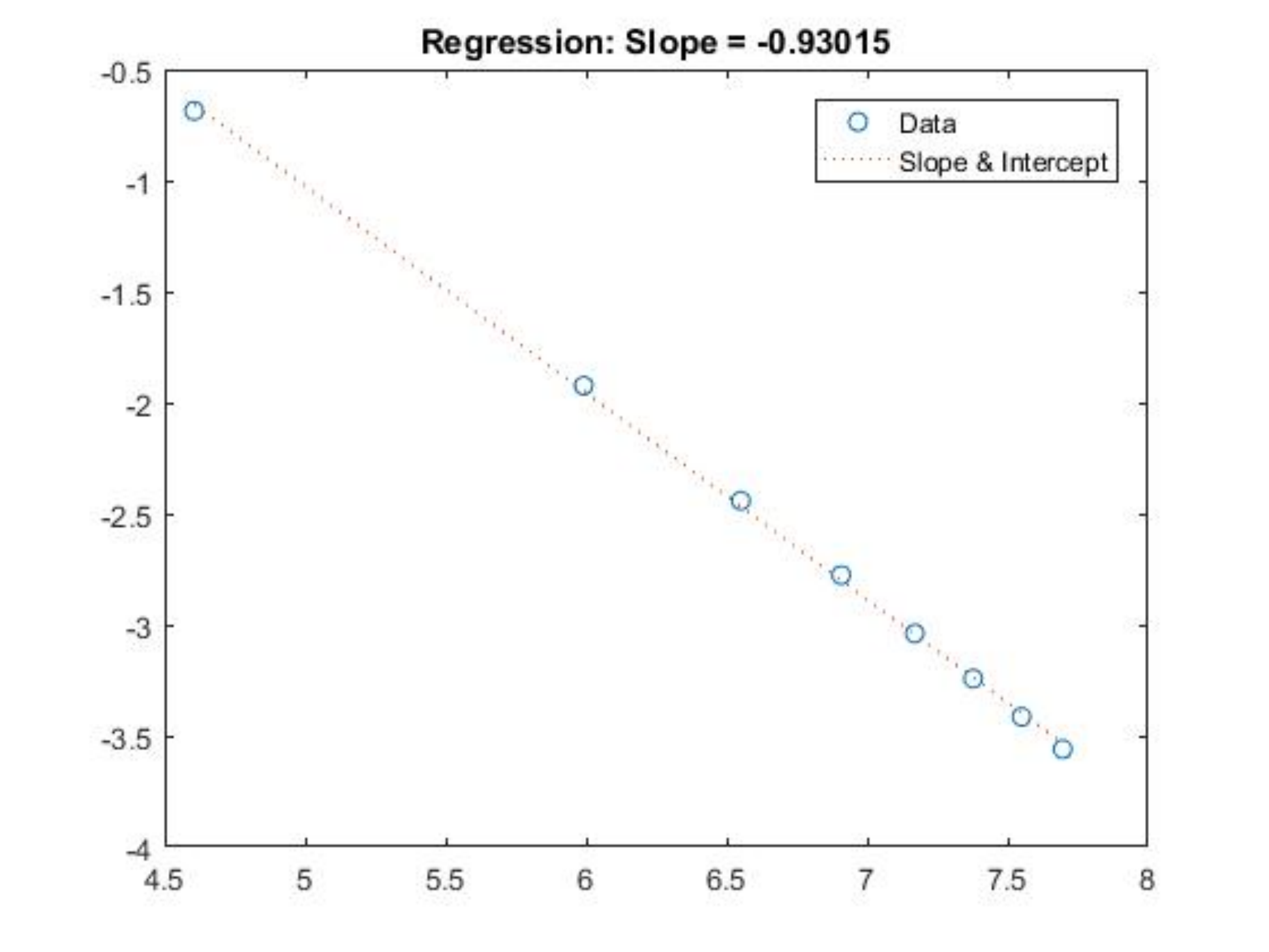}
  \caption{Case (i). Regression of $\log(\hat{E})$ w.r.t. $\log(N)$. Data: $\hat{E}$ when $N$ varies from $100$ to $2200$ with step size $300$. Parameters: $n = 100$, $T = 1$, $\beta =2$, $\sigma= 1$, $\lambda=5$, $x_0 = 1$, $p = 1/2$, $L = 1000$.}
\end{figure}

\begin{figure}[h!]
	 \label{fig:trois}
	\centering
	\includegraphics[scale=0.5]{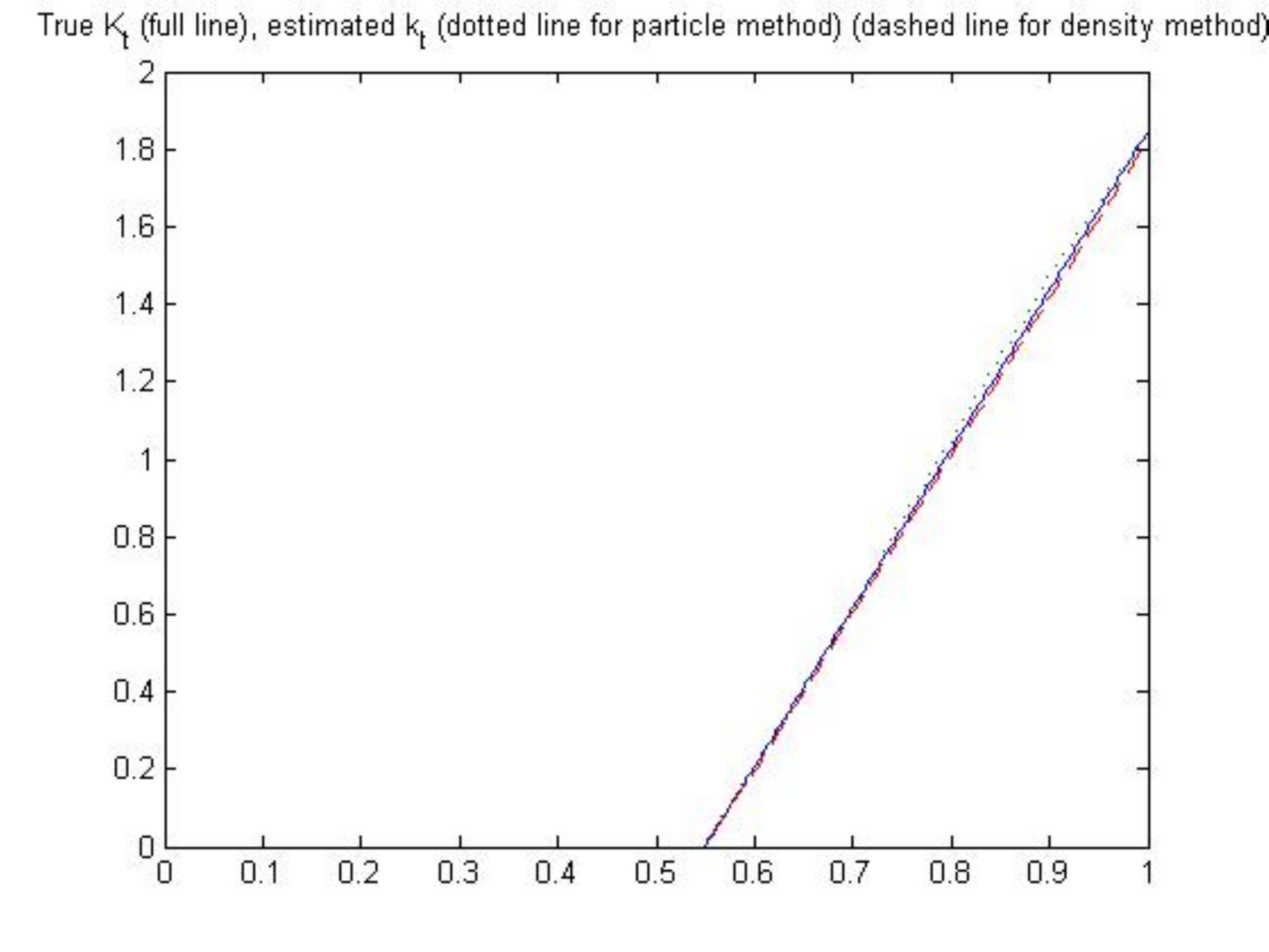}
	\caption{Case (ii). Parameters: $n = 500$, $N = 10000$, $T = 1$, $\beta=0$, $a = 3$, $\gamma=1$, $\eta =1$, $\lambda=2$, $x_0 = 4$, $p = 1$.}
\end{figure}

\begin{figure}[h!]
	\label{fig:quatre}
  \centering
  \includegraphics[scale=0.5]{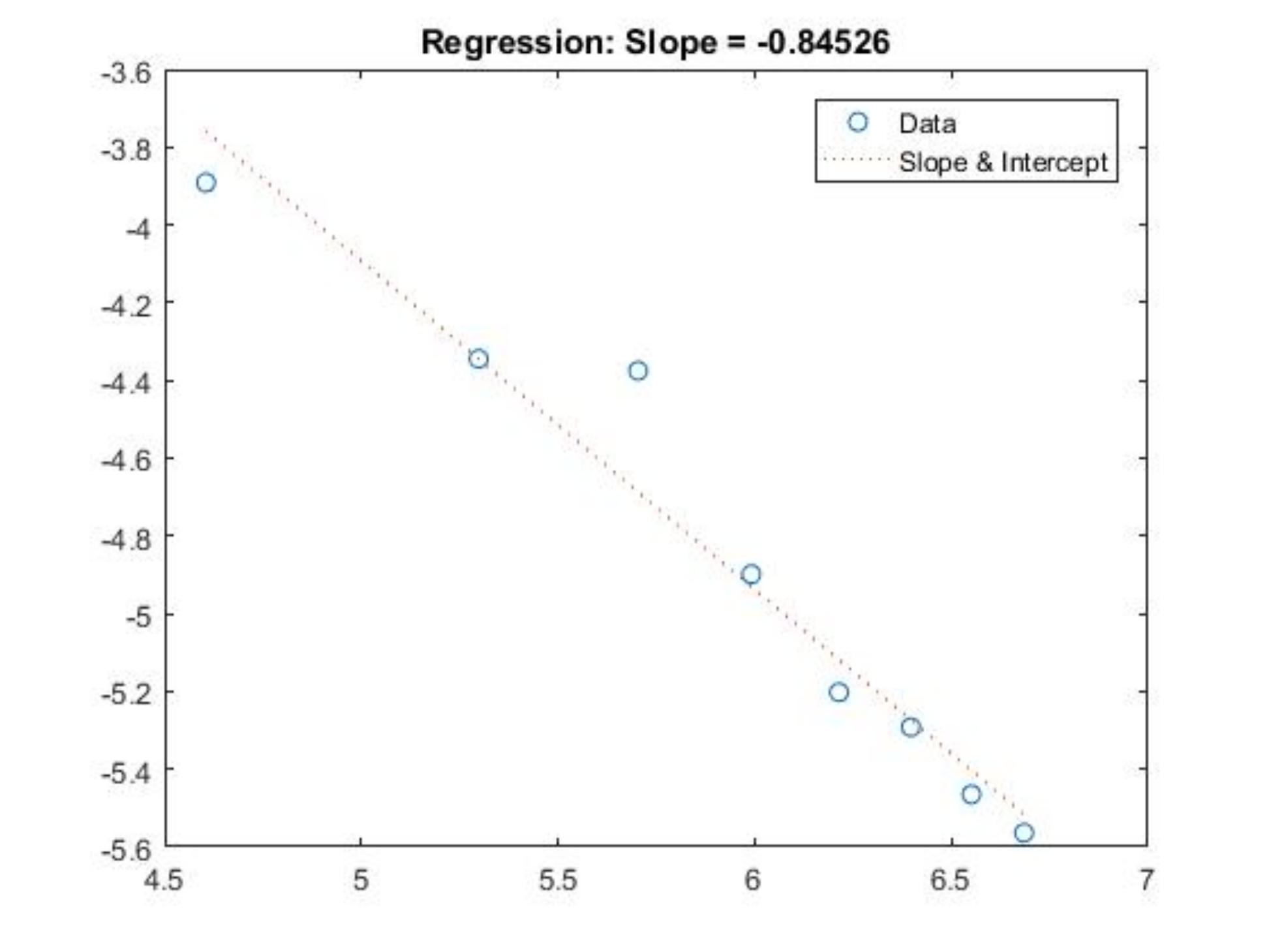}
  \caption{Case (ii). Regression of $\log(\hat{E})$ w.r.t. $\log(N)$. Data: $\hat{E}$ when $N$ varies from $100$ to $800$ with step size $100$. Parameters: $n = 1000$, $T = 1$, $\beta =0$, $a=3$, $\gamma=1$, $\eta= 1$, $\lambda=2$, $x_0 = 4$, $p = 1$, $L = 1000$.}
\end{figure}

\begin{figure}[h!]
	\label{fig:cinq}
	\centering
	\includegraphics[scale=0.5]{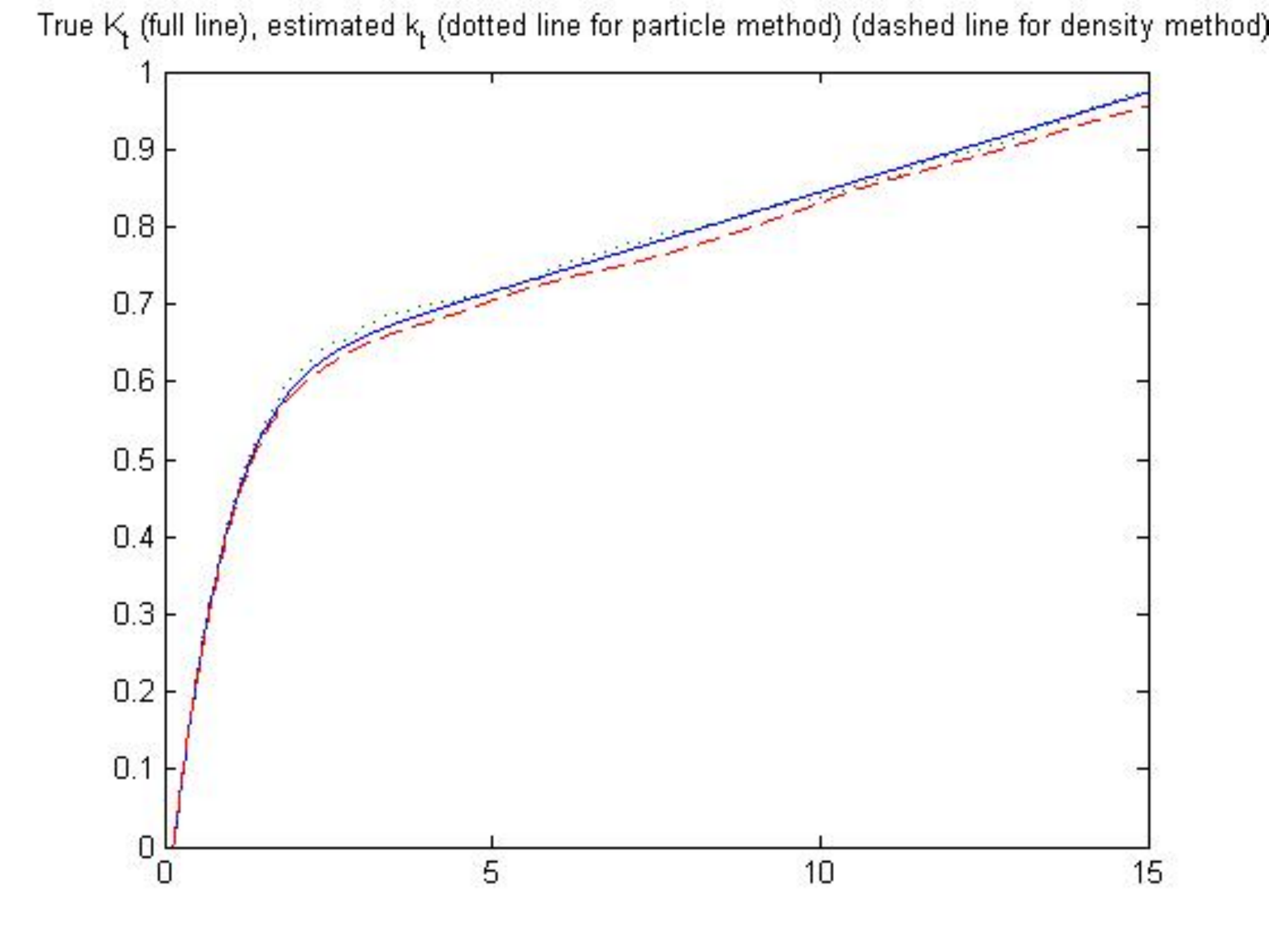}
	\caption{Case (iii). Parameters: $n = 1000$, $N = 100000$, $T = 15$, $\beta =10^{-2}$, $\sigma = 1$, $p =\pi/2$, $\alpha = 0.9$, $a=10^{-2}$,  $x_0$ is the unique solution of $x+\alpha\sin (x)-p=0$ plus $10^{-1}$.}
\end{figure}

\appendix

\section{Appendices}
\subsection{Proof of Lemma \ref{lemmacont_generale}}
Let $s$ and $t$ in $[0,T]$ such that $s\leq t$.\\
Firstly, we suppose that $\varphi$ is a continuous function with compact support. In this case, there exists a sequence of Lipschitz continuous functions $\varphi_n$ with compact support which converges uniformly to $\varphi$.
Therefore, by using Proposition \ref{propriete_2}, we get
\begin{equation*}
\begin{aligned}
|\e[\varphi(X_t)]-\e[\varphi(X_s)]|&\leq |\e[\varphi(X_t)]-\e[\varphi_n(X_t)]|+|\e[\varphi_n(X_t)]-\e[\varphi_n(X_s)]|+|\e[\varphi_n(X_s)]-\e[\varphi(X_s)]|\\&
\leq \e[|(\varphi-\varphi_n)(X_t)|]+C_n\e[|X_t-X_s|]+|\e[(\varphi_n-\varphi)(X_s)]|\\&
\leq 2\e[\parallel\varphi_n-\varphi\parallel_\infty]+C_n(\e[|X_t-X_s|^2])^{1/2}\\&
\leq 2\e[\parallel\varphi_n-\varphi\parallel_\infty]+C_n|t-s|^{1/2}.
\end{aligned}
\end{equation*}
Thus, we obtain that
$$\limsup_{t\rightarrow s}|\e[\varphi(X_t)]-\e[\varphi(X_s)]|\leq 2\e[\parallel\varphi_n-\varphi\parallel_\infty].$$
This result is true for all $n\geq1$, so we deduce that
$$\limsup_{t\rightarrow s}|\e[\varphi(X_t)]-\e[\varphi(X_s)]|=0,$$
then we conclude the continuity of the function $t\longmapsto\e [\varphi(X_t)]$.\\
\\
Secondly, we consider the case that $\varphi$ is a continuous function such that $$\forall x\in\mathbb{R}, \exists C\in\mathbb{R},~~\varphi(x)\leq C(1+|x|^p).$$
We define a sequence of functions $\varphi_n$, such that for all $n\geq1$ and $x\in\mathbb{R}$, 
$$
\varphi_n(x)=\varphi(x)\theta_n(x)
$$ 
with $\theta_n$ smooth such that 
\begin{equation*}
    \theta_n(x)=
    \begin{cases}
      1 & \text{if}\ |x|\leq n \\
      0 & \text{if}\ |x|> n+1
    \end{cases}
  \end{equation*}
 Based on this definition, $\varphi_n$ is a continuous function with compact support. Then we get 
 \begin{align*}
 |\e[\varphi(X_t)]-\e[\varphi(X_s)]| & 
 \leq \e\left[\left|\varphi-\varphi_n\right|(X_t)\right]  + \left|\e[\varphi_n(X_t)]-\e[\varphi_n(X_s)]\right|  + \e\left[\left|\varphi-\varphi_n\right|(X_s)\right]\\
 & \leq \e\left[2|\varphi(X_t)|\ind_{|X_t|> n}\right]+\left|\e[\varphi_n(X_t)]-\e[\varphi_n(X_s)]\right|+ \e\left[ 2|\varphi(X_s)|\ind_{|X_s|> n}\right] \\
 & \leq C\e\left[ \left(1+\sup_{t\leq T}|X_t|^p\right)\ind_{\sup_{t\leq T}|X_t|> n}\right]+\left|\e[\varphi_n(X_t)]-\e[\varphi_n(X_s)]\right|.
 \end{align*}
  Thus, by using the first part of this Lemma, we obtain that
  $$
	\limsup_{t\rightarrow s}|\e[\varphi(X_t)]-\e[\varphi(X_s)]|\leq C\e\left[\left(1+\sup_{t\leq T}|X_t|^p\right)\ind_{\sup_{t\leq T}|X_t|> n}\right].
	$$
  This result is true for all $n\geq1$, then by using the dominated convergence theorem, we deduce that
  $$
	\limsup_{t\rightarrow s}|\e[\varphi(X_t)]-\e[\varphi(X_s)]|=0,
	$$
  and we conclude the continuity of the function $t\longmapsto\e [\varphi(X_t)]$.


\end{document}